\documentclass{amsart}

\usepackage{amsmath,multicol}
\usepackage{amssymb}
\usepackage{amsthm}
\usepackage[all]{xypic}
\usepackage{amsfonts}
\usepackage{fullpage}
\usepackage{verbatim,color}
\usepackage{url}
\usepackage{hyperref}

\numberwithin{equation}{section}
\theoremstyle{plain}
\newtheorem{theorem}[equation]{Theorem}
\newtheorem{proposition}[equation]{Proposition}
\newtheorem{corollary}[equation]{Corollary}
\newtheorem{lemma}[equation]{Lemma}
\newtheorem{question}[equation]{Question}

\newtheorem{hypothesis}[equation]{Hypothesis}
\newtheorem{example}[equation]{Example}

\theoremstyle{definition}
\newtheorem{definition}[equation]{Definition}
\newtheorem{notation}[equation]{Notation}

\theoremstyle{remark}
\newtheorem{remark}[equation]{Remark}

\newcommand{\beq}{\begin{equation}}
\newcommand{\eeq}{\end{equation}}

\newcommand{\st}{\left\vert\right.}

\newcommand{\bbar}[1]{\overline{#1}}

\DeclareMathOperator{\Hom}{{Hom}}

\DeclareMathOperator{\Ext}{{Ext}}

\DeclareMathOperator{\Tor}{Tor}
\DeclareMathOperator{\Aut}{{Aut}}

\DeclareMathOperator{\pd}{pd}

\DeclareMathOperator{\GK}{GKdim}

\DeclareMathOperator{\Div}{Div}

\DeclareMathOperator{\Ann}{Ann}
\DeclareMathOperator{\sat}{sat}
\DeclareMathOperator{\Pic}{Pic}

\newcommand{\mc}{\mathcal}
\newcommand{\mf}{\mathfrak}

\newcommand{\kk}{{\Bbbk}}

\newcommand{\ZZ}{{\mathbb Z}}

\newcommand{\NN}{{\mathbb N}}

\newcommand{\mb}{\mathbb}
\newcommand{\wt}{\widetilde}

\newcommand{\sC}{\mc{C}}

\newcommand{\sF}{\mc{F}}

\newcommand{\sI}{\mc{I}}
\newcommand{\sJ}{\mc{J}}
\newcommand{\sK}{\mc{K}}
\newcommand{\sL}{\mc{L}}
\newcommand{\sM}{\mc{M}}

\newcommand{\sO}{\mc{O}}
\newcommand{\sP}{\mc{P}}

\newcommand{\divb}{\mathbf{b}}
\newcommand{\divd}{\mathbf{d}}
\newcommand{\dive}{\mathbf{e}}
\newcommand{\divc}{\mathbf{c}}

\newcommand{\divx}{\mathbf{x}}
\newcommand{\divy}{\mathbf{y}}
\newcommand{\divz}{\mathbf{z}}
\newcommand{\divr}{\mathbf{r}}
\newcommand{\divw}{\mathbf{w}}

\DeclareMathOperator{\fd}{fd-\!}
\DeclareMathOperator{\rgr}{gr-\!}
\DeclareMathOperator{\lgr}{\!-gr}

\DeclareMathOperator{\rGr}{Gr-\!}
\DeclareMathOperator{\rmod}{mod-\!}
\DeclareMathOperator{\lmod}{\!-mod}

\DeclareMathOperator{\lqgr}{\!-qgr}

\DeclareMathOperator{\rQgr}{Qgr-\!}
\DeclareMathOperator{\rqgr}{qgr-\!}

\DeclareMathOperator{\rTors}{Tors-\!}

\DeclareMathOperator{\rtors}{tors-\!}

\DeclareMathOperator{\GKdim}{GKdim}

\newcommand{\spe}{sporadic}

\newcommand{\Rbar}{\overline{R}}

  \newcommand{\ppe}{{\  \buildrel\bullet\over =  \  }}
\newcommand{\wh}{\widehat}

\DeclareMathOperator{\tors}{tors}
\newcommand{\tg}{\tors_g}
\DeclareMathOperator{\coh}{coh}

\begin{document}
\title[Noncommutative blowups of elliptic algebras]{Noncommutative blowups of elliptic algebras}
\author{D. Rogalski,  S. J. Sierra, and J. T. Stafford}
\address{(Rogalski)
Department of Mathematics, UCSD, La Jolla, CA 92093-0112, USA. }
\email{drogalsk@math.ucsd.edu}
 \address{(Sierra) School of Mathematics,
University of Edinburgh, Edinburgh EH9 3JZ, U.K.}
\email{s.sierra@ed.ac.uk}
\address{(Stafford) School of Mathematics,  The University of Manchester,   Manchester M13 9PL,
U.K..}
\email{Toby.Stafford@manchester.ac.uk}

\thanks{The first author is partially supported by NSF grants DMS-0900981 and DMS-1201572.}
\thanks{The second author was  partially   supported by an NSF Postdoctoral Research
Fellowship, grant DMS-0802935.}
\thanks{The third
   author is a Royal Society Wolfson Research Merit Award Holder.}
%\date{\today}

 \keywords{Noncommutative projective geometry,  noncommutative surfaces, 
 Sklyanin algebras,  noetherian  graded rings,
noncommutative  blowing~up}
  \subjclass[2000]{14A22, 14H52, 16E65, 16P40, 16S38, 16W50, 18E15}
   \begin{abstract}  
We develop a ring-theoretic approach for blowing up many noncommutative projective surfaces.  
Let $T$ be an elliptic algebra (meaning that, for some central element $g\in T_1$, 
$T/gT$  is a twisted homogeneous coordinate ring of an elliptic curve $E$ at an infinite order automorphism).
Given  an effective divisor  $\divd$    on $E$  whose degree is not too big, 
we construct a blowup $T(\divd)$ of $T$ at $\divd$  and show that it is also an elliptic algebra. 
Consequently it has many good properties:  for example, it is strongly noetherian, Auslander-Gorenstein,   and has a balanced dualizing complex.  We also show that  the ideal structure of $T(\divd)$  is quite rigid.
Our results generalise those of the first author in \cite{Rog09}.  

In the companion paper \cite{RSS2}, we apply our results to classify orders in (a Veronese subalgebra of) a generic cubic or quadratic Sklyanin algebra. 
\end{abstract}
 
 \maketitle   
 \tableofcontents
 
  \clearpage

\section{Introduction}\label{INTRO}  
Blowing up a  projective surface at a point, or more formally  applying a monoidal transformation, is a 
 fundamental operation in commutative algebraic geometry.  The aim of this paper is to study noncommutative 
 analogues of this concept.

There are two basic approaches to noncommutative projective algebraic geometry:  one may study graded rings which have some of the properties
 of homogeneous coordinate rings of projective varieties, or one may study categories which are similar to the category of coherent sheaves on 
 a projective variety.  Here we take the first approach.  
We show, in a unified and ring-theoretic way, that there is a noncommutative analogue of the blowing-up process, and we establish many of
 the properties of the resulting algebras. 
 Our results are applied in the companion paper \cite{RSS2} to classify orders in (Veronese subalgebras of)  Sklyanin algebras.

Fix an algebraically closed ground field $\kk$ and let $T = \kk \oplus T_1 \oplus T_2 \oplus \dots$ be a connected graded $\kk$-algebra 
that is a domain.   We say that $T$ is an {\em elliptic algebra} if there exists a central element $g \in T_1$ so that $T/gT$ is a {\em twisted
 homogeneous coordinate ring} $B(E, \sM, \tau)$ on an elliptic curve $E$, where $\sM$ is  ample and $\tau \in \Aut(E)$ has infinite order.  
(See Section~\ref{BACKGROUND} for background material on these coordinate rings.)  
Elliptic algebras are noncommutative analogues of  anticanonical coordinate rings of projective (Fano or del Pezzo) surfaces; the best-known 
examples are Veronese subalgebras of generic Sklyanin algebras.
As we will see in a moment, they are extremely well-behaved.

Let $T$ be an elliptic algebra.  We define the {\em degree} of $T$ to be $\mu = \deg \sM = \dim T_1 - 1$. 
Our main result is that we can use the points of $E$ to ``blow up'' $T$ and create further elliptic algebras.
More specifically, let  $\divd$ be an effective divisor on $E$, with $\deg \divd < \mu$.  
We define a graded subalgebra $T(\divd)$ of $T$, and show that $T(\divd)$ has the   properties
described in the following theorem. Here, if $R$ is a noetherian graded domain   its {\em graded quotient ring} $Q_{gr}(R)$ is formed 
by inverting all nonzero homogeneous elements  (\cite[Corollary~8.1.21]{MR}
and \cite[C.I.1.6]{NV}).  The ring $R$ is a \emph{maximal order} if there exists no order $R\subsetneqq R'\subseteq Q_{gr}(R)$
with $aR'b\subseteq R$ for some nonzero $a,b\in R$.
 Undefined terms  in the theorem can be found in Section~\ref{BACKGROUND} or in \cite{Rog09}.

\begin{theorem}\label{ithm:main}
{\rm (Theorem~\ref{thm:Rnoeth})}
Let $T$, $\mu$, $\divd$ be  as above, and write   $d = \deg \divd$. Set $R = T(\divd)$.
Then:
\begin{enumerate}
\item $R$ is an elliptic algebra of degree $\mu-d$, with $R/gR \cong B(E, \sM(-\divd), \tau)$.  
As such,   $R$ is strongly noetherian, Auslander-Gorenstein, Cohen-Macaulay, and a maximal order.  In addition, $R$ 
satisfies the Artin-Zhang $\chi$ conditions  
and possesses a  balanced dualizing complex. 
\item $R$ has Hilbert series $h_{R}(t) = \sum (\dim_{\kk} R_i)t^i = h_{T}(t) - d/(1-t)^3$, and  $Q_{gr}(R) = Q_{gr}(T)$.
\item $R(\divd)$ is generated in degree $1$ if $\mu - d \geq 2$.  If $\mu - d = 1$, then  $R$ is generated in degrees
 $\leq 2$, and $R^{(2)}$ is generated in degree $1$.
\item {\rm (Proposition~\ref{prop:iterate})} Further,  $T(\divd)$ can be obtained by iterating a series of  
one-point blowups. 
\end{enumerate}
\end{theorem}
\noindent
 
  For    reasons why it is appropriate to regard $T(\divd)$ as a blowup, see Section~\ref{BLOWUP}.
In particular  various basic properties  of a commutative blowup $\pi: \wt{X} \to X$  do have noncommutative analogues.    
For one-point blowups, we prove:

\begin{proposition}\label{iprop:subtle}
{\rm (Proposition~\ref{prop:subtle})}
Let $T$ be as above, with $\deg T = \mu > 1$.  Let $p \in E$ and set $R = T(p)$. 
Then there is a finitely generated right $R$-module $L$  so that 
$T/R \cong \bigoplus_{i \geq 1} L[-i]$ as right $R$-modules.  
Further, $L$ is  {\em $g$-torsionfree} in the obvious sense, and has Hilbert series $h_L(t)  = 1/(1-t)^2$. 
\end{proposition}

\begin{remark}
It is natural to consider $L$ as \emph{the exceptional line} for the noncommutative blowup $T(p) $ of $ T$.   
\end{remark}

The elliptic algebras that we are most interested in are blowups of Veronese subalgebras of generic Sklyanin algebras.  
We  show that these particular elliptic algebras have a relatively rigid ideal structure.

\begin{theorem}\label{ithm:sporadic}
{\rm(Theorem~\ref{8-special})}
 Let  $S$ be a generic cubic or quadratic Sklyanin algebra, with central element $g \in S_{\gamma}$. 
Let $T = S^{(\gamma)}$ and $\mu = \deg T$. 
Let $\divd$ be an effective divisor on $E$ with $\deg \divd < \mu$.  Then $T(\divd)$ has
 a {\em minimal \spe\ ideal $I$} in the following sense:  $T/I$ has Gelfand-Kirillov dimension 
 $\leq 1$, and for any graded ideal $J \subseteq T$ with $\GKdim T/J = 1$, we have 
 $J_n \supseteq I_n $ for $n \gg 0$.  
\end{theorem}

The above results generalise  those of the first author in \cite{Rog09}.  
In that paper, Theorems~\ref{ithm:main} and \ref{ithm:sporadic}~are proved where $T=S^{(3)}$ is the 3rd Veronese of a generic quadratic Sklyanin algebra, and $\deg \divd \leq 7 = \mu-2$.  
Allowing $\deg \divd = \mu -1$ requires one to overcome  technical difficulties, as the resulting rings are no longer  generated in degree 1.   
The rings $T(\divd)$ also appear in Van den Bergh's memoir \cite{VdB}.  
Van den Bergh, however, works mostly  on the categorical level; in contrast, our approach is   more elementary and is more amenable to computation.

The main motivation of our work  is to be able to answer the following question:
\begin{question}\label{intro-qu}
 {\rm (\cite[Question~1.3]{Rog09})}
  Can one classify all  maximal orders $R$ inside a generic (quadratic) Sklyanin algebra $S$ with 
$Q_{gr} (R) = Q_{gr} (S^{(n)} )$ for some $n$? 
\end{question}
\noindent
(We remark that maximal orders, the noncommutative analogues of integrally closed rings, are a natural class to consider in
such  a classification problem.)
In the companion paper \cite{RSS2}, we solve Question~\ref{intro-qu} for $n\in 3\mathbb{N}$.  
In particular, we prove:  

\begin{theorem}\label{ithm:maxnoeth}
{\rm(\cite[Theorem~1.2, Corollary~1.4, and Theorem~5.25]{RSS2})}
Let $T = S^{(3)}$ be as above, and let $R $ be a maximal order in $T$ such that $R \not \subseteq \kk + gT$.  Then $R$ is noetherian and is obtained from $T$ by the analogue of a  blowup of $T$ at a possibly non-effective divisor $\divx$.  Further, $R$ is equivalent to some  $T(\divd)$, where $\divd$ is effective with $\deg \divd \leq 8 = \mu-1$.
\end{theorem}

\noindent
In \cite{RSS2}, we  also obtain a classification (up to Veronese subalgebras) of arbitrary orders in $S^{(3n)}$.

One of the authors' ultimate goals is to classify all algebras with the same graded quotient ring as a Sklyanin algebra.
Such a result would contribute significantly to solving    
the fundamental problem  of the  classification of noncommutative projective surfaces \cite{Ar}. 
From a ring-theoretic perspective, one can  frame the problem as:  
\emph{what are 
the graded domains with GK-dimension  $3$}?

   Let $R$ be such an algebra.  Then $Q_{gr}(R) \cong D[z, z^{-1}; \tau]$ 
for some division ring $D$, called the {\em function skewfield} of $R$,  and some  $\tau \in \Aut_{\kk}(D)$.  
The study of  birationally equivalent,  irreducible projective varieties   can be interpreted  ring-theoretically as   the study of different graded 
domains $R$ with the same function skewfield.
As Artin has argued in  \cite{Ar}, one hopes that algebras (or at least maximal orders)  in the same birational class are related through some 
sort of blowing up and blowing down processes.  By Theorem~\ref{ithm:maxnoeth}, this is true for 
 orders in $T$.
   
\subsection*{Organisation of the paper}  
In outline,  Section~\ref{BACKGROUND} contains general background material and basic properties of elliptic algebras while   Section~\ref{FUNCTORIAL}  reviews the necessary  technical machinery  from \cite{VdB}.
  This leads directly in Section~\ref{SOMEIDEALS} to the construction of the blowups $T(\divd)$.  We prove Theorem~\ref{ithm:main} and other basic properties of these rings in Section~\ref{BLOWUP}.  Section~\ref{RIGHT-LEFT} contains   results relating left and right ideals of $T(\divd)$.
  In Section~\ref{DIVISORS} we construct the exceptional line module of a one-point blowup.
In Section~\ref{8POINTS2} we study the  ideal theory of $T(\divd)$ and  prove Proposition~\ref{iprop:subtle} and Theorem~\ref{ithm:sporadic}.  
 Finally, in Section~\ref{SPECIAL} 
we determine circumstances when  a blowup $T(\divd)$  has no nontrivial \spe\ ideals.   
At the end of the paper we provide an index of notation.

\subsection*{Acknowledgements}
    Part of this material is based upon work supported by the National Science Foundation under Grant No. 0932078 000, while the authors were in residence at the Mathematical Science Research Institute (MSRI) in Berkeley, California, during the Spring 
   semester of 2013.   During this trip, Sierra was also partially supported by the Edinburgh Research Partnership in Engineering and Mathematics, and Stafford was partially supported by the Clay Mathematics Institute
   and  Simons Foundation.    The authors gratefully acknowledge the   support of all these organisations.
 
%%%%%%%%%%%%%%%%%%%%%%
\section{Basic properties} \label{BACKGROUND}

 We begin by reviewing some definitions and basic background material for the paper.

 Throughout, fix an 
algebraically  closed base field $\kk$.  Let $E$ be a projective $\kk$-scheme
(which will always be a (necessarily smooth) elliptic curve), with     an invertible sheaf $\mc{L}$
  and   an automorphism $\tau$.   
For any sheaf $\mc{F}$ on $E$ we  write
$\mc{F}^{\tau} = \tau^*(\mc{F})$   for the pullback along $\tau$.
Using this data we can define a \emph{TCR} or  \emph{twisted homogeneous coordinate ring}
$B(E, \mc{L}, \tau) = \bigoplus_{n \geq 0} H^0(E, \mc{L}_n)$.  Here, 
 $\mc{L}_n = \mc{L} \otimes \mc{L}^{\tau} \otimes \dots \otimes  \mc{L}^{\tau^{n-1}}$ and multiplication 
  is defined     by $f \star g  = f \otimes (\tau^m)^*(g)$ for $f \in B_m$ and  $g \in B_n$.   
We will typically use bold notation 
$\divd$ for   divisors on $E$ and write  $\divd^{\tau} = \tau^{-1}(\divd)$, 
so that, if   $\mc{L} \cong \mc{O}_E(\divd)$, then $\mc{L}^{\tau} = \mc{O}_E(\divd^{\tau})$.   
Similarly,  we write $\divd_n = \divd + \divd^{\tau} + \dots + \divd^{\tau^{n-1}}$.
Some of the special features of  $B(E, \mc{L}, \tau)$   are described in  \cite[Section~3]{Rog09}.

 In this paper we always assume that $|\tau|=\infty$, which  has the following useful   consequences.

 \begin{lemma}  {\rm(\cite[Lemma~3.1]{Rog09})}
\label{lem:sec-mult}
Let $E$ be an elliptic curve with an  automorphism $\tau$ of infinite order.  
  \begin{enumerate}\item Given invertible sheaves $\mc{L}, \mc{M}$ on $E$, the natural map
\[
H^0(E, \mc{L}) \otimes H^0(E, \mc{M}) \to H^0(E, \mc{L} \otimes \mc{M})
\]
is surjective if $\deg \mc{L} \geq 2$ and $\deg \mc{M} \geq 2$, except in the case when $\deg \mc{L}   = 2$ and 
$\mc{L} \cong \mc{M}$. 
\item If $\deg \mc{L}\geq 2$  then $B(E, \mc{L}, \tau)$ is generated in degree $1$.  \qed
\end{enumerate}
 \end{lemma}

In this paper we will consider the following special class of rings.  
  Here, an $\NN$-graded $\kk$-algebra $R$ is {\em connected graded} (cg) if $R_0 = \kk$ and $\dim_\kk R_n < \infty $ for all $n$.
  
\begin{hypothesis}
\label{hyp:main}
Let $T$ be a   cg $\kk$-algebra which is a domain 
with a central   element $g \in T_1$  such that  
$T/gT \cong B  = B(E, \mc{M}, \tau)$ for an elliptic curve $E$, invertible sheaf $\mc{M}$ of 
degree $\mu \geq 2$, and infinite order automorphism $\tau$.  
By \cite[Lemma~2.7]{RSS2} (or by Lemma~\ref{lem:sec-mult} and induction),  $T$ is  finitely  generated in degree~$1$.

Given   $X \subseteq T$, we always write  $\bbar{X}=(X+gT)/gT \subseteq B(E, \sM, \tau)$.  
\end{hypothesis}

One of the main ways of producing such a ring $T$ is to start 
with either the   \emph{quadratic Sklyanin algebra} or the   \emph{cubic Sklyanin algebra} $S$
of dimension 3.   This  has a  central element $g\in S_\gamma$ (for $\gamma=3$, respectively $4$)
 such that $S/gS \cong B(E, \mc{L}, \sigma)$. 
We always assume that $|\sigma|=\infty$, in which case, by
  a slight abuse of notation, we call $S$ a \emph{generic Sklyanin algebra.}
 Then the Veronese ring  $T=S^{(\gamma)}$  satisfies Hypothesis~\ref{hyp:main}.
   The hypothesis also holds for Veronese rings of generic Stephenson's algebras \cite{Ste}.
See \cite[Examples~2.2]{RSS2} for the explicit definitions of these  algebras.

The rings $T$ in Hypothesis~\ref{hyp:main} automatically have   many good  properties, as  we next show.   This requires the following definitions. 
The remaining terms in the proposition are 
not defined since they do not play a significant r\^ole in this paper;  the relevant definitions can be found, for example, in 
 \cite[Section~2]{Rog09}.
 
 \begin{definition}\label{AG-defn}
A ring $A$ is called  \emph{Auslander-Gorenstein} if it has finite injective dimension and satisfies the
 {\it Gorenstein condition}: if $p<q$ 
are non-negative integers and  $M$ is a finitely generated $A$-module, then
$\Ext_A^p(N,\,A)=0$ for every submodule $N$ of 
$\Ext_A^q(M,\,A)$.
Set
 $j(M)= \min\{ r | \Ext^r_A(M,A)\not= 0\}$ for the \emph{homological grade} of $M$.  
An Auslander-Gorenstein ring $A$ of finite GK-dimension is
 called \emph{ Cohen-Macaulay} (or CM), provided 
 that  $j(M)+\GKdim(M)=\GKdim(A)$ holds
for every finitely generated
 $A$-module $M$.    
Finally, an $A$-module $M$ is called \emph{Cohen-Macaulay} (or CM) if $\Ext^i_A(M,A) = 0$ for all $i \neq j(M)$.
\end{definition}

\begin{proposition}\label{prop:foo}
Let $S$ be a cg $\kk$-algebra which is a domain with a central element $g \in S_{\gamma}$ for some $\gamma \geq 1$, such that 
$S/gS \cong B(E, \mc{L}, \sigma)$ for some elliptic curve $E$, invertible sheaf $\mc{L}$ with $\deg \mc{L} \geq 1$, and 
infinite order automorphism $\sigma$.  

Then $S$ is strongly noetherian, Auslander-Gorenstein, CM and a maximal order.
Also, $S$ satisfies the Artin-Zhang $\chi$ conditions,  has cohomological dimension $2$ and possesses a balanced dualizing complex. 
\end{proposition}
\begin{proof}  
Use the proofs of \cite[Theorems~6.3  and 6.7]{Rog09} (the hypothesis assumed there that 
$\deg \mc{L} \geq 2$ is not needed in the proofs). 
The one exception to this assertion  is the  fact that $S$ has a  balanced dualizing complex,
which follows from   \cite[Theorem~6.3]{VdB4}.
\end{proof}

Next, we review   some of the  categories used in the paper.   
Let $R$ be a cg noetherian $\kk$-algebra and write $\rGr R$ for the category of $\mb{Z}$-graded right $R$-modules, with  
$\Hom_{\rGr R}(-,-)$ being  graded homomorphisms of degree zero. 
Let $\rTors R\subseteq \rGr R$ be the full subcategory of \emph{torsion} modules; thus modules  $M$
with the property that for every $m \in M$, $m A_{\geq n} = 0$ for
some $n \geq 1$.     Write $\rgr R$ for the full subcategory of $\rGr R$
of finitely generated  $R$-modules, with torsion subcategory $\rtors R$.
The quotient category $\rqgr R = \rgr R/\rtors R$ has the same objects as $\rgr R$, 
but morphisms are given by $\Hom_{\rqgr R}(\pi(M), \pi(N)) = \lim_{n \to \infty} \Hom_{\rgr R}(M_{\geq n}, N)$, 
for $M,N \in \rgr R$.  The quotient category $Y = \rQgr R = \rGr R/\rTors R$ is defined similarly, and $\rqgr R$ may be 
identified with the subcategory of noetherian objects in $Y$; see \cite[p. 234-5]{AZ}.
The quotient functor 
$\pi: \rGr R \to Y $ has a right adjoint, the \emph{section functor} $\omega: Y \to \rGr R$ \label{section-defn}   given 
by $\omega(\mc{F}) = \bigoplus_{n \in \mb{Z}} \Hom_Y(\mc{O}_Y, \mc{F}[n])$. Here $\mc{O}_Y  = \pi(R)$ 
and $[n]$ is the shift functor on $Y$, induced by the shift of graded modules $M \mapsto M[n]$
 where $M[n]_m = M_{n+m}$.

If $J  $ is a  graded right ideal of $R$, write  
$J^{\sat} = \big\{ x \in R \st x R_{\geq n} \subseteq J\ \text{for some}\ n \geq 0 \big\}$ 
for  the \emph{saturation} of $J$.  It is the unique largest  right ideal  $K\supseteq J$ such that $\dim_\kk(K/J)<\infty$. 
We call $J$ \emph{saturated} if $J = J^{\sat}$.
 
\begin{lemma}
\label{lem:T-props}
Assume that $T$ satisfies Hypothesis~\ref{hyp:main}.
  \begin{enumerate}
 \item The natural map $T \to \omega(\pi(T))$ is an isomorphism of rings.
 \item Given any right ideal $J$ of $T$, we have $\omega(\pi(J))  = J^{\sat} \subseteq T$.
\end{enumerate}
\end{lemma}
\begin{proof}
(1) By \cite[(3.12.3)]{AZ}, for any $M \in \rgr T$ there is an exact sequence
\[
0 \to t(M) \to M \to \omega \pi(M) \to \lim_{n \to \infty} \Ext^1(T/T_{\geq n}, M) \to 0,
\]
where $t(M)$ is the largest finite-dimensional submodule of $M$.  Apply this equation with $M = T$.
Then $t(T) = 0$ since $T$ is a domain. By Proposition~\ref{prop:foo} and  \cite[Lemma~4.11(2)]{RSS2},  
  $T$ has no non-trivial extensions by finite-dimensional modules   and 
so $\lim_{n \to \infty} \Ext^1(T/T_{\geq n}, T)$ = 0.  Thus the natural map $T \to \omega \pi(T)$ is an isomorphism.
As in \cite[Section~4]{AZ},  it is also a ring homomorphism. 

(2) Since $\omega$ is left exact, $J' = \omega(\pi(J)) \subseteq T$ is a right ideal of $T$
   and $J'/J$ is a torsion module by applying 
the   exact sequence  to $M = J$.     It is easy to see that $\omega(\pi(J))$ 
is already saturated, so that $J' = J^{\sat}$. 
\end{proof}

The following  definitions will be used frequently.
\begin{definition}
Let $S$ be  a cg $\kk$-algebra with homogeneous central element $g \in S$.  Let $M$ be a (right) $S$-module.  The \emph{$g$-torsion submodule} of $M$ is 
$\operatorname{tors}_g(M) = \{m \in M \st m g^n  = 0\ \text{for some}\ n \geq 1 \}$.  The module $M$ is 
\emph{$g$-torsion} if $\operatorname{tors}_g(M) = M$ and \emph{$g$-torsionfree} if $\operatorname{tors}_g(M) = 0$.
A graded vector subspace $V$ of $S$ is {\em $g$-divisible}\label{g-div-defn}
 if $V \cap Sg = Vg$; if $V$ is a right ideal, this is equivalent to $\operatorname{tors}_g(S/V) = 0$. 
\end{definition}

\begin{definition}\label{point-defn}
We say that $M \in \rgr R$ is a \emph{point module} if it has Hilbert series 
$h_M(t) = 1/(1-t)$.  A \emph{shifted point module} 
is a module of the form $M[n]$ for some integer $n$ and point module $M$.  We remark that point modules are traditionally assumed to be cyclic, necessarily generated in degree $0$. However, 
that assumption is not convenient for us, since it is really only appropriate 
for algebras generated in degree one.  
\end{definition}

Gelfand-Kirillov (GK) dimension will be the main dimension function in this paper.   
Recall that $M \in \rgr R$ is \emph{$\alpha$-pure} if \label{pure-defn}
$\GKdim(N) = \alpha$ for all nonzero submodules $N \subseteq M$, and 
\emph{$\alpha$-critical} if $\GKdim(M/N) < \GKdim(M)$ 
for all such $N$.  
The next lemma clarifies the relationship between   $1$-critical modules and  point modules over
$B(E,\sL,\sigma)$.
Here, we  write $\mc{I}_p$ for the ideal sheaf defining a closed  point $p\in E$, so that $ \mc{O}_E/\mc{I}_p=\kk(p)$, the corresponding skyscraper sheaf.
Since $\rqgr B \sim \coh E$, the category of coherent sheaves on $E$ (see \cite[Theorem~1.3]{AV}),
 there is also a corresponding simple object $\mc{O}_p \in \rqgr B$. Explicitly, 
$\mc{O}_p = \pi(P(p))$, where $P(p)= \bigoplus_{n \geq 0} H^0(E, \kk(p) \otimes \mc{L}_n)$
 is the right point module of $B$ corresponding to $p$.

\begin{lemma}
\label{lem:pt-crit}
Let $B = B(E, \mc{L}, \sigma)$, where $E$ is   elliptic, $\deg \mc{L} \geq 1$, and $\sigma$ has infinite order.
\begin{enumerate}
\item If $M \in \rgr B$ is a point module, then $M_{\geq n}$ is $1$-critical for some $n \geq 0$.
\item For any $p \in E$, $P(p) $ is a (not necessarily cyclic) point module for $B$.
If $M \in \rgr B$ is $1$-critical, then $\pi(M) = \sO_p$ in $\rqgr B$ for some $p\in E$.
\item If $\deg \mc{L} \geq 2$, then $M \in \rgr B$ is  $1$-critical if and only if it is a shifted cyclic point module.

\end{enumerate}
\end{lemma}
\begin{proof}

(1) Since $\GKdim(M)=1$ and $M$ has a unique largest finite dimensional submodule, 
$M_{\geq n}$ is 1-pure for $n\gg 0$. As such, $M_{\geq n}$ has a cyclic  1-critical submodule, say $Q$
(use the proof of  \cite[Proposition~6.2.20]{MR}).
But $  \dim_{\kk} Q_m = 1$ for all $m \gg 0$ by  \cite[Corollary~3.7(2)]{RSS2}. 
So $M_{\geq m}=Q_{\geq m}$ is 1-critical for all such $m$.

(2)  This follows from the equivalence of categories $\coh E \to \rqgr B$ and the fact that $\pi(M)$ is a simple object in $\rqgr B$
(see \cite[Corollary~3.7(1)]{RSS2} and its proof).

(3) This  is standard. In more detail, assume that $M$ is 1-critical.
By \cite[Corollary~3.7(2)]{RSS2},    $\dim_{\kk} M_n \leq 1$ for all $n \in \mb{Z}$, with equality for $n \gg 0$.  
As  $B$ is now  generated in degree $1$ (see Lemma~\ref{lem:sec-mult}) and $M$ is 1-pure, 
this means there is $j \in \mb{Z}$ such that $\dim_{\kk} M_n = 1$ for $n \geq j$ and $\dim_{\kk} M_n = 0$ for $n < j$.  
The module $M$ is  then forced to be cyclic, generated in degree $j$, and so $M$ is a shifted cyclic point module.
\end{proof}

%%%%%%%%%%%%%%%%%%%%%%%
\section{A categorical equivalence for right ideals}\label{FUNCTORIAL}

For the whole of this section, fix a ring $T$ satisfying Hypothesis~\ref{hyp:main} and its 
associated notation,   in particular $E$ is an elliptic curve with an  infinite order 
automorphism $\tau$.

The main goal of this section is to review a categorical equivalence developed by Van den Bergh and Van Gastel \cite{VV}, which 
gives a formally useful way to describe and work with cyclic $T$-modules of GK-dimension~$1$ and their defining ideals.  
We then use this to associate certain important right $T$-ideals to lists of divisors on $E$, and begin to analyse 
their properties.

  Let $X = \rqgr T$ and $Y = \rqgr B$, for $B = T/gT$.  
Then $Y  $ is a subcategory of $X$ with $ \coh E\sim Y$ via the map  
$\mc{F} \mapsto \pi (\bigoplus_{n \geq 0} H^0(E, \mc{F} \otimes \mc{M}_n))$ (see \cite{AV}).  
Thus, we have the main hypothesis of \cite[Chapter~5]{VdB}; namely, in the language of that paper, $Y$ is a 
commutative CM curve embedded as a divisor in the noetherian quasi-scheme~$X$.   Moreover, \cite{VdB}, Hypothesis $(*)$ (which demands that objects in $\rqgr B$ have finite injective dimension in
$\rqgr T$) holds automatically; see \cite{VdB}, Hypothesis $(*')$ and the discussion there.

One would like to understand the category $\sC_f$ of  finite length objects in $X$ whose composition factors lie in $Y$.
Following \cite[Section~5]{VdB}, we let $\mc{C}_{f,p}$ be the subcategory of $X$ consisting of finite-length objects whose Jordan-Holder 
quotients are all simple objects of the form $\mc{O}_{\tau^i(p)}$ for various $i\in\ZZ$.  
One may show that $\dim_\kk \Ext^1_{\rqgr B}(\mc{O}_p, \mc{O}_q) = \delta_{p,q}$, but 
 $\dim_\kk \Ext^1_{\rqgr T}(\mc{O}_p, \mc{O}_q) = 1$ when $p = \tau(q)$.
   In some sense, these two types of extensions generate the entire category $\sC_{f,p}$, 
in a way which will be made precise by the next theorem.

Let $C_p$ be the ring of $\mb{Z} \times \mb{Z}$ lower triangular matrices with entries in $k [[ x]]$.   Write $e_{ij}$ for the matrix~units in $C_p$. 
Thus     $C_p$ is a locally noetherian ring, with idempotents $e_i=e_{ii}$  and   
$C_p \cong \prod_i e_i C_p$, as right $C_p$-modules.   Each $e_i C_p$ has a unique simple factor module $S_i$ and we write  
$S_i = C_p/\mf{n}_i$ for the appropriate ideal $\mf{n}_i$.
  Let $N$ be the matrix with a $1$ in the $(i, i-1)$-position for all $i \in \mb{Z}$, and $0$'s elsewhere. 
Then $N$ is normal in $C_p$ and the ideal $N C_p$ consists of all matrices in $C_p$ which are $0$ along the main diagonal.  Let $\rmod C_p$ 
be the category of   finitely generated right $C_p$-modules, and  
$ \fd C_p$  the category of $C_p$-modules of finite $\kk$-dimension.

The next result from \cite{VV,VdB},  which  relates  $\mc{C}_{f,p}$  to  $\fd C_p$, forms the starting point to this section.  It provides
 a very convenient way of analysing   the graded $T$-modules of GK-dimension~1.

\begin{theorem}  
Let $X = \rqgr T$ for $T$ satisfying Hypothesis~\ref{hyp:main},  and keep the above  notation.  
Fix $p \in E$.    
\label{thm:vdb}
\begin{enumerate}
\item There is an equivalence of categories $( \wh{-} )_p: \mc{C}_{f,p} \to   \fd C_p$.
We have $(\wh{\mc{O}_{\tau^i(p)}})_p = S_i$ for each $i \in \mb{Z}$.

\item  More generally, there is an exact functor $( \wh{-} )_p: X = \rqgr T \to \rmod C_p$, 
which restricts on $\mc{C}_{f,p}$ to the equivalence of Part (1).
One has $(\wh{\sO_X})_p = C_p$, and for any right ideal $L$ of $T$, $\wh{\pi(Lg)}_p = \wh{\pi(L)}_p N$.

\item Given a noetherian object $\mc{F} \in X$, $( \wh{-} )_p$ defines a one-to-one 
correspondence between subobjects  $\mc{G}$ of $\mc{F}$ such that 
$\mc{F}/\mc{G} \in \mc{C}_{f,p}$ and $C_p$-submodules $H$ of $\wh{\mc{F}}_p$ 
such that $\wh{\mc{F}}_p/H \in \fd C_p$.   Under this correspondence, 
$\mc{F}/\mc{G} \in Y$ if and only if 
$\wh{\mc{F}}_p N \subseteq  \wh{\mc{G}}_p \subseteq \wh{\mc{F}}_p$.  Similarly, under this correspondence 
$\mc{F}/\mc{G}$ is a direct sum of copies of $\mc{O}_{\tau^i(p)}$ 
if and only if $\wh{\mc{F}}_p \mf{n}_i \subseteq \wh{\mc{G}}_p  \subseteq \wh{\mc{F}}_p$.  
\end{enumerate}
\end{theorem}
\begin{proof}
This is a combination and restatement of several results in \cite{VdB}, which are themselves 
a generalisation of the original result from \cite{VV}.
These results are stated in the language of topological rings and modules, since the ring $C_p$ is 
a topological ring with the co-finite-length topology.  We will not  use this language, other than to explain 
how our result follows from the results in \cite{VV, VdB}.
 
(1)  This follows from \cite[Theorem~1.1]{VV} or from the more general result \cite[Theorem~5.1.4]{VdB}.

(2)  The functor $( \wh{-} )_p$ from Part (1) is extended to an exact functor $X = \rqgr T \to \rmod C_p$ 
in \cite[Section~5.3]{VdB}.  There the image is taken to lie in the category %$\operatorname{PC}(C_p)$ 
of pseudo-compact $C_p$-modules
as defined in \cite[Section~4]{VdB}, but it is also noted that the objects in the image of the functor 
are finitely generated.
The fact that $(\wh{\sO_X})_p = C_p$ is \cite[Lemma~5.3.3]{VdB}, and the last statement follows from
 %the generalization of \cite[Theorem~5.1.4(4)]{VdB} to this extended functor
  \cite[Theorem~5.3.1]{VdB}.  

(3)  By \cite[Corollary~5.3.5]{VdB}, $( \wh{-} )_p$ induces a bijection between subobjects 
$\mc{G}$ of $\mc{F}$ such that 
$\mc{F}/\mc{G} \in \mc{C}_{f,p}$ and open subobjects of $( \wh{\mc{F}} )_p$.  However, such open subobjects 
are the same  as co-finite-dimensional $C_p$-submodules    (use \cite[Lemma~4.2]{VdB}
 and the fact that      finite length $C_p$-modules are finite-dimensional).   

Now by the equivalence of Part (1), $(\mc{F}/\mc{G})$ is a direct sum of copies of $\mc{O}_{\tau^i(p)}$ if and only if 
$\wh{\mc{F}}_p/\wh{\mc{G}}_p$ is a direct sum of copies of $S_i$.  
The latter condition  is equivalent to $\wh{\mc{F}}_p \mf{n}_i \subseteq \wh{\mc{G}}_p  \subseteq \wh{\mc{F}}_p$,
as desired. 
The other statement is proved similarly:  use the fact that an object is in $\mc{C}_{f, p} \cap Y$  
  if and only if its image under $(\wh{-})_p$ is a $C_p/(N)$-module \cite[Theorem~5.1.4(4, 5)]{VdB}.
\end{proof}

Using Theorem~\ref{thm:vdb} we can associate some   important right ideals of $T$   to sequences of divisors on $E$.

\begin{definition}  
A finite sequence of effective divisors $(\divd^0, \dots \divd^{k-1})$ on $E$ is 
called an \emph{allowable divisor layering} if it satisfies the condition
\beq \label{allowable}
\tau^{-1}(\divd^{i-1}) \geq \divd^{i} \quad \text{for all $1 \leq i \leq k-1$.}
\eeq
We adopt the convention that $\divd^m=0$ for all $m \geq k$, and we use the notation $\divd^\bullet = (\divd^0, \dots, \divd^{k-1})$ 
to represent the divisor layering as a whole.
If $q\in E$, let $\mb O(q) = \{ \tau^i (q) \st i \in \ZZ \}$ be the $\tau$-orbit of $q$.  
\end{definition}

\begin{definition}
\label{def:J}  
Assume Hypothesis~\ref{hyp:main}, and let $\divd^\bullet$ be an allowable divisor layering.  
 To  $\divd^\bullet$ we will associate a subobject $\mc{J}(\divd^\bullet)$ of $\mc{O}_X$, where $X = \rqgr T$, and 
a  saturated right ideal $J(\divd^\bullet)$ of $T$.  

  Suppose first that each $\divd^i$ is supported  on $\mb{O}(p)$. 
Then $\divd^i = \sum_j a_{j + i, j} \tau^{j}(p)$ 
for some unique  integers $a_{k, \ell}\geq 0$ and 
  we   define   $\mf{J}= \mf{J}(\divd^i)\subseteq C_p$ by putting $(x)^{a_{k, \ell}}$ in the $(k, \ell)$-spot
of the lower triangular matrix.   To see that $\mf{J}$ is  a right ideal, note that, by condition \eqref{allowable}, 
 $a_{k, \ell} \leq a_{k, \ell+1}$ for  all $k, \ell$.  Under the correspondence given in 
Theorem~\ref{thm:vdb}(3), $\mf{J} = \wh{\mc{J}}_p$ for a subobject $\mc{J} = \mc{J}(\divd^\bullet)$ of 
$\mc{O}_X$ such that $\mc{O}_X/\mc{J}$ has finite length with composition factors among the 
$\mc{O}_{\tau^i(p)}$.   Thus   $J(\divd^\bullet) = \omega(\mc{J}(\divd^\bullet))$ is a saturated right ideal  of $T$.  
 
Given a general allowable divisor layering, choose  representative points $p_1, \dots p_s$ on distinct $\tau$-orbits, 
such that every point involved in $\divd^\bullet$ lies in    $\mb{O}({p_k})$, for some $k$.  
For each $p_k$ we define $\divd^i_k = \divd^i|_{\mb{O}(p_k)}$.  Then $\divd^\bullet_k$ is  an allowable divisor layering
and we put $\mc{J}(\divd^\bullet) = \bigcap_k \mc{J}(\divd^\bullet_k)$ and $J(\divd^\bullet) = \bigcap_k J(\divd^\bullet_k) 
= \omega(\bigcap_k \mc{J}(\divd^\bullet_k))$.
\end{definition}

The following result explains the idea behind the name divisor layering: each divisor $\divd^j$ tracks which points (with multiplicities)
occur in the $j^{\text{th}}$ layer $Mg^j/Mg^{j+1}$ of the module $M = T/J(\divd^\bullet)$.

\begin{lemma} 
\label{lem:layer}
Let $\divd^{\bullet}$ be an allowable divisor layering and let $J = J(\divd^{\bullet})$ and $M = T/J$.
\begin{enumerate}
\item If $M^j = Mg^j/Mg^{j+1}$, then as objects in $\rqgr B$ we have 
\[
\pi(M^j) \cong \pi \bigg(\bigoplus_{n \geq 0} H^0(E, (\mc{O}_E/\mc{O}_E(-\divd^j)) \otimes \mc{M}_n) \bigg).
\]
In particular, the divisor $\divd^j$ determines  
%which simple objects in $\rqgr B \simeq \coh E$, with multiplicity, 
%occur in a composition series of this finite length object, or equivalently 
the (tails of) point modules that occur in a filtration of $M^j$.% by tails of point modules.
\item $(\overline{J})^{\sat} = \bigoplus_{n \geq 0} H^0(E, \mc{M}_n(-\divd^0)).$ 
\item If $\divd^{\bullet} = (\divd)$ has length $1$, then 
$J(\divd) = \bigoplus_{n \geq 0} \{ x \in T_n \st \bbar{x} \in H^0(E, \sM_n(-\divd)) \}$.  
\end{enumerate}
\end{lemma}

\begin{proof} 
(1) By definition, $\mc{J}(\divd^{\bullet}) = \bigcap_k \sJ(\divd_k^{\bullet})$ where $\divd^i_k = \divd^i \vert_{\mb{O}({p_k})}$ and 
the orbits   $\mb{O}(p_k)$ are distinct.  Also, the composition factors of $\mc{O}_X/\mc{J}(\divd_k^{\bullet})$ are among the 
$\mc{O}_{\tau^i(p_k)}$ by construction.  Thus the   $\mc{J}(\divd_k^{\bullet})$ must be pairwise comaximal inside $\mc{O}_X$, 
and so $\pi(M) = \mc{O}_X/\mc{J}(\divd^{\bullet}) \cong \bigoplus_k \mc{O}_X/\mc{J}(\divd_k^{\bullet})$.  
Since $\bigoplus_k \mc{O}_E/\mc{O}_E(-\divd^j_k) \cong \mc{O}_E/\mc{O}_E(-\divd^j)$, we 
can  reduce to the case that the divisor layers $\divd^i$ are all supported on a single $\tau$-orbit $\mb{O}(p)$.

Now since $\pi(M) \in X$ corresponds  via Theorem~\ref{thm:vdb}(1) to 
$(\wh {\pi(M)})_p = C_p/\mf{J}(\divd^\bullet)$, 
it suffices to study the filtration of the latter module by $C_p/(N)$-modules.   
From the definition 
of the ideal $\mf J = \mf{J}(\divd^\bullet)$, one   sees that writing $\divd^i = \sum b_j \tau^j(p)$ gives  
$(\mf J + N^i C_p)/(\mf J+ N^{i+1}C_p) \cong     
C_p/(\prod_j \mf{n}_j^{b_j})  \cong \bigoplus_j C_p/\mf{n}_j^{b_j}$.  
Using Theorem~\ref{thm:vdb}, again, the object in $\rqgr B \sim \coh(E)$ corresponding to this 
object under the equivalence $( \wh{-} )_p$ is $\mc{O}_E/\mc{O}_E(- \divd^i)$.

(2) Applying Part (1) with $j = 0$, shows that $M/gM = T/(gT + J) = \overline{T}/\overline{J}$ satisfies 
$$\pi(\overline{T}/\overline{J}) \cong \pi(\bigoplus_{n \geq 0}
 H^0(E, \mc{O}_E/\mc{O}_E(-\divd^0) \otimes \mc{M}_n)).$$ 
This shows that $\overline{J}$ is equal in large degree to
 $\bigoplus_{n \geq 0} H^0(E, \mc{O}_E(-\divd^0) \otimes \mc{M}_n)$.  But 
the latter right ideal is easily seen to be saturated.

(3)  Since $J$ is saturated and there is only one layer, $Tg \subseteq J$.  Now the result follows from Part~(2).
\end{proof}

We always use the partial ordering on   divisors on $E$ where $\divc \leq \divd$ means that $\divd - \divc$ is effective.  
This induces a $\max$ and $\min$ operation on divisors, where explicitly 
$\max(\sum a_p p, \sum b_p p) = \sum \max(a_p, b_p) p$ and similarly for the $\min$. 
 These    operations are extended to divisor layerings coordinatewise.
It is   easy to see that the association of right ideals to divisors respects the lattice structures, as follows.

\begin{lemma}
\label{lem:lattice}
Let $\divc^\bullet$ and $\divd^\bullet$ be allowable divisor layerings, and set $\dive^\bullet = \max(\divc^\bullet, \divd^\bullet)$
and $\divb^\bullet = \min(\divc^\bullet, \divd^\bullet)$.
Then $\mc{J}(\divc^\bullet) \cap \mc{J}(\divd^\bullet) = \mc{J}(\dive^\bullet)$ and 
$\mc{J}(\divc^\bullet) + \mc{J}(\divd^\bullet) = \mc{J}(\divb^\bullet)$.  
Furthermore, $J(\divc^\bullet) \cap J(\divd^\bullet) = J(\dive^\bullet)$ and $(J(\divc^\bullet) + J(\divd^\bullet))^{\sat} = J(\divb^\bullet)$.
\end{lemma}
\begin{proof}
First, it is routine to reduce to the case that $\divc^{\bullet}$, $\divd^{\bullet}$, and $\dive^{\bullet}$ 
are supported on a single orbit $\mb{O}(p)$.

Considering the right ideals $\mf{J}(\divd^\bullet)$ and $\mf{J}(\divc^\bullet)$ of $C_p$ as in Definition~\ref{def:J}, we have $\mf{J}(\divc^\bullet) \cap \mf{J}(\divd^\bullet) = \mf{J}(\dive^\bullet)$ because the intersection of two right ideals in $C_p$ of the form $( (x)^{a_{i,j}})$ is clearly formed by taking the maximum of the $a_{i,j}$ at each spot.
Similarly, $\mf{J}(\divc^\bullet) + \mf{J}(\divd^\bullet) = \mf{J}(\divb^\bullet)$ because the sum of two such right ideals 
is formed by taking the minimum of the $a_{i,j}$ at each spot.

The functor $(\wh{-})_p$ of Theorem~\ref{thm:vdb} is exact, and so the correspondence of 
Theorem~\ref{thm:vdb}(3) must preserve intersections and sums of subobjects of a fixed object.   Thus 
 $\mc{J}(\divc^\bullet) \cap \mc{J}(\divd^\bullet) = \mc{J}(\dive^\bullet)$ and 
$\mc{J}(\divc^\bullet) + \mc{J}(\divd^\bullet) = \mc{J}(\divb^\bullet)$, by the definition of the $\mc{J}$'s.
  Applying the section functor $\omega$ gives the last statement.
\end{proof}

Right ideals of the form $J(\divd^\bullet)$ represent only some of the right ideals $J$ for which $T/J$ 
is $g$-torsion with $\GKdim(T/J)=1$.  To see this, use Theorem~\ref{thm:vdb} to pass to some $C_p$.  
Note that $C_p/\mf{n}_0 \mf{n}_{-1}\cong\left(\begin{smallmatrix}\kk&0\\ \kk & \kk\end{smallmatrix}\right)$,  and so     there are  
 infinitely many right ideals $\mf{K}\subset C_p$ such that $C_p/\mf{K}$ is an essential extension of $S_{-1}$ by $S_0$. However only one of these can correspond to some $J(\divd^\bullet)$;
in this case Lemma~\ref{lem:layer} implies that $\divd^\bullet =(p,\tau^{-1}(p))$.

For the rest of this section, we consider certain functors on $X = \rqgr T$ and their actions on the   $\mc{J}(\divd^\bullet)$.

\begin{definition}\label{F-defn}
Fix a point $p$,  and consider Theorem~\ref{thm:vdb} applied to the $\tau$-orbit of this point.
Given $q= \tau^i(p)$, we define a functor $F_{q}: X \to X$ as follows.  For an object $\mc{G} \in X$, let  $F_q(\mc{G})$ be the smallest subobject $\mc{G}'$ of $\mc{G}$ such that $\mc{G}/\mc{G}'$ is isomorphic to a direct sum of copies of the simple object $\mc{O}_{q}$.  That is, by Theorem~\ref{thm:vdb}, $\mc{G}'$ is the subobject of $\mc{G}$ corresponding under the functor $( \wh{-} )_p$ to $\wh{\mc{G}}_p \mf{n}_i$.  
Given a morphism $f \in \Hom(\mc{G}, \mc{H})$,  we define $F_q(f) = f \vert _{F_q(\mc{G})}  \in \Hom(F_q(\mc{G}), F_q(\mc{H}))$.  
To see that this is well-defined, note that $\mc{G}/f^{-1}(F_q(\mc{H}))$ embeds in $\mc{H}/F_q(\mc{H})$ and 
thus is also a direct sum of copies of $\mc{O}_q$.   Thus $f^{-1}(F_q(\mc{H}))$ contains the smallest such subobject $F_q(\mc{G})$ of $\mc{G}$, which implies that $f(F_q(\mc{G})) \subseteq F_q(\mc{H})$.

If $G = F_{q_1} \circ \dots \circ F_{q_n}$ is a composition of such functors,  write $\tau^j(G)$ for the functor $F_{\tau^j(q_1)} \circ \dots \circ F_{\tau^j(q_n)}$.   
\end{definition}

\begin{notation}\label{not:Gd}
Given an effective divisor $\divd$  on $E$, we define a functor $G_{\divd}: X \to X$  as follows.  
  Break $ \divd$ up by orbit so that $\divd = \sum_{j=1}^n \divd_j$, where $\divd_j$ is supported entirely on $\mb{O}(p_j)$ for some $p_j \in E$.   
  If needed, reindex $\mb O (p_j)$ so that $\divd_j = \sum_{i=0}^m b_i \tau^i(p_j)$  and define  $G_j =  F_{\tau^m(p_j)}^{b_m}\circ \dots \circ F_{\tau(p_j)}^{b_1} \circ F_{p_j}^{b_0}$.  Finally, set $G_\divd = G_1 \circ G_2 \circ \dots \circ G_n$.
 \end{notation}
 
\begin{lemma}\label{lem-Fcomp}
 Let $\divz^\bullet = (\divz^0, \dots, \divz^{k-1})$ be an allowable divisor layering, and let $\mc{J} = \mc{J}(\divz^\bullet)$.  Let $\divd$ be an effective divisor on $ E$.
Then $G_{\divd}(\mc{J})$ is defined by the allowable divisor layering $\divx^\bullet = (\divx^0, \dots, \divx^k)$, 
where 
\[ \divx^i = \begin{cases} \divz^0 +\divd & \text{if $i=0$} \\
               \min( \divz^i + \divd, \tau^{-1}(\divz^{i-1})) & \text{if $1 \leq i \leq k$.}
              \end{cases}\]

\end{lemma}
\begin{proof}
We prove this by induction.  If $\divd=p$ is a single point, this 
 is  basically a restatement of \cite[Proposition~5.2.2]{VdB}.  
 To see this, note that, by construction, $\divx^\bullet$ is allowable; note that since $\divz^i \leq \tau^{-1} \divz^{i-1}$, therefore $\divz^i+p \leq \tau^{-1}(\divz^{i-1}+p)$ if and only if $\divz^i +p \leq \tau^{-1} (\divz^{i-1})$.  We must  prove that $\divx^\bullet$ defines $F_p\mc{J}$.  

It is clear from the definition of $F_p$ that to compute $F_p(\mc{J})$ we only need to consider the orbit of $p$:  that is, without loss of generality we may assume that $\divz^0$ is supported on the orbit of $p$.  Define natural numbers $a_{\ell,m}$ by writing
$\divz^i = \sum_j a_{j+i,j} \tau^j(p)$.  Then $\wh{\mc{J}}_p = \prod_{\ell \in \ZZ} P_\ell$, where
$P_\ell \subseteq e_\ell C_p$ is the infinite row vector defined by putting $(x)^{a_{\ell,m}}$ in the $(\ell,m)$ spot.
As remarked in the definition of $\mc{J}( \divd^\bullet)$, each $P_\ell$ is  a right ideal of $C_p$.

The object $F_p \mc{J}$ corresponds under $( \wh{-} )_p$ to $\wh{\mc{J}}_p \cdot \mf{n}_0 = \prod_\ell P_\ell\cdot \mf{n}_0$.  
But $P_\ell \cdot \mf{n}_0$ is easy to describe:  namely, by \cite[Proposition~5.2.2]{VdB} or a direct computation, 
$P_\ell \cdot \mf{n}_0$ is defined by $(x)^{b_{\ell,m}}$ where
\[ b_{\ell,m} = \begin{cases} a_{\ell,m} & \text{if $m \neq 0$} \\
                  a_{\ell,0}+1 & \text{if $a_{\ell,0} < a_{\ell,1}$ or if $\ell=0$} \\
		  a_{\ell,0} & \text{if $\ell \geq 1$ and $a_{\ell,0} = a_{\ell,1}$.}
                 \end{cases}\]
                 
     Now let $a'_{\ell,m}$ be defined by $\divx^i = \sum_j a'_{i+j,j} \tau^j(p)$.  The $a'_{\ell,m}$ satisfy
     $a'_{0,0} = a_{0,0}+1$, $a'_{\ell,m} = a_{ \ell,m} $ if $m \neq 0$, and if $\ell>0$ then 
\[ a'_{\ell,0} = \begin{cases} a_{\ell,0}+1 & \text{if $\tau^{-1}(\divz^{\ell-1}) \geq \divz^\ell +p$, that is, 
 if $a_{\ell,1} > a_{\ell,0}$} \\ a_{\ell,0} & \text{otherwise}. \end{cases} \]
We see that $a'_{\ell,m} = b_{ \ell,m}$.  It follows from Theorem~\ref{thm:vdb} that $F_p \mc{J} = \mc{J}( \divx^\bullet)$.

Now assume that  $\deg \divd  \geq 2$.  Choose $q \in E$ and an effective divisor $\divd'$ on $E$ so that 
 $\divd = q + \divd'$ and  $\tau^{-1}( \divd')$ does not contain $q$.  Note that $G_{\divd} = F_q \circ G_{\divd'}$.  By induction,
$G_{\divd'} \mc{J}$ is defined by the allowable divisor layering 
\[ \divy^i = \begin{cases} \divz^0 + \divd' & \text{if $i = 0$} \\
\min(\divz^i + \divd', \tau^{-1}( \divz^{i-1})) & \text{$ 1 \leq i \leq k$}.
\end{cases}\]
Applying the one-point case, we find that 
$ G_{\divd}(\mc{J})$ is defined by
\[ \divx^i = \begin{cases}
\divz^0 + \divd' +q = \divz_0 + \divd & \text{if $i = 0$}\\
\min( \divz^1 + \divd, \tau^{-1}(\divz^0)+q, \tau^{-1} (\divz^0 + \divd')) & \text{if $i=1$} \\
  \min(\divz^i+\divd, \tau^{-1}(\divz^{i-1})+q, \tau^{-1}(\divz^{i-1}+\divd'), \tau^{-2} (\divz^{i-2})) & \text{if $2 \leq i \leq k+1$},
  \end{cases} \]
  and that this is allowable.  
  Since $\tau^{-1}(\divd')$ does not meet $q$ and $\divz^{i-1} \leq \tau^{-1} (\divz^{i-2})$, this simplifies to 
\[ \divx^i = \begin{cases} \divz^0 + \divd & \text{if $i=0$} \\
\min(\divz^i+\divd, \tau^{-1} (\divz^{i-1})) & \text{if $1 \leq i \leq k+1$}. \end{cases}\]
Finally, since $\divz^k = 0$, therefore  $\divx^{k+1}=0$.  The lemma is proved.
\end{proof}

We also extend the previous result to calculate the result of applying  certain compositions of $G_{\divd}$'s.

\begin{lemma}\label{lem-Fcomp2}
Let $\divd$ be an effective divisor, and let $\divd_n = \divd+ \tau^{-1}\divd + \dots + \tau^{-(n-1)}\divd$ if $n \geq 1$.  If $n \leq 0$, 
define $\divd_n = 0$.  
 Let $\mc{J} = \mc{J}(\divz^\bullet)$, where $\divz^\bullet$ is an allowable divisor sequence.

 Define $G_{\divd}$ as in Notation~\ref{not:Gd}.  
Then for any $n \in \NN$, the object $(\tau^{-(n-1)}G_{\divd} \circ \dots \circ \tau^{-1}G_{\divd} \circ G_{\divd}) (\mc{J})$  is equal to
 $\mc{J}(\divw^\bullet)$,   where
\[ \divw^i = \min_{j=0}^i (\tau^{-j}(\divz^{i-j}+\divd_{n-j})).\]
\end{lemma}
\begin{proof}
The case $n=1$ is Lemma~\ref{lem-Fcomp}, since $\tau^{-1}(\divz^{i-1}) \leq \tau^{-2}(\divz^{i-2}) \leq \dots$.  Assume by induction that
$(\tau^{-(n-2)}G_{\divd} \circ \dots \circ  G_{\divd})(\mc{J})$ is defined by $\divx^{\bullet}$, where
\[ \divx^i = \min_{j=0}^i (\tau^{-j}(\divz^{i-j}+\divd_{n-1-j})).\]
Applying Lemma~\ref{lem-Fcomp}, we obtain that $\divw^0 = \divz^0  +\divd_{n-1}+\tau^{-(n-1)}\divd = \divz^0 + \divd_n$, 
and for $i \geq 1 $ we have
\begin{align*}
\divw^i \ = & \min(\divx^i+\tau^{-(n-1)}\divd, \tau^{-1}(\divx^{i-1})) \\  
&=\ \min\left(\min_{j=0}^i(\tau^{-j}(\divz^{i-j}+\divd_{n-1-j})+\tau^{-(n-1)}(\divd)),
\min_{j=0}^{i-1}(\tau^{-j-1}(\divz^{i-1-j}+\divd_{n-1-j}))\right).
\end{align*}
Since $\tau^{-j}(\divd_{n-1-j})+\tau^{-(n-1)}(\divd) = \tau^{-j}(\divd_{n-j})$, this expression reduces to the desired expression for $\divw^i$.
\end{proof}

\section{Certain right ideals of   \texorpdfstring{$T$}{LG}   and their Hilbert series}\label{SOMEIDEALS}

We continue to assume throughout this section that $T$ is a ring satisfying Hypothesis~\ref{hyp:main}.

Starting in this section, we work more directly in the graded ring $T$.   
We begin by defining an important special case of the right ideals $J( \divd^\bullet)$, which will be crucial in describing the blowups $T(\divd)$ later.
Recall   the notation $\divd_n = \divd + \tau^{-1}(\divd) + \dots  + \tau^{-n+1}(\divd)$ from Section~\ref{BACKGROUND}. 

\begin{definition}\label{def:M(k,D)}
Fix an effective divisor $\divd$ on $E$ and for each $k \geq 1$, let 
$M(k,\divd)$  be the ideal $J( \divd^\bullet)$ for the (necessarily allowable) divisor layering $\divd^\bullet= (\divd^0, \divd^1, \dots, \divd^{k-1})$, 
where 
\[
\divd^0 = \divd_k, \ \ \divd^1 = \tau^{-1}(\divd_{k-1}), \  
\dots, \ \divd^{k-1} = \tau^{-k+1}(\divd).
\] 
By convention, if $k \leq 0$, then $M(k, \divd) =T$, regardless of $\divd$.
 
Note that  $M(k, \divd)$ and $J(\divd^\bullet)$ depend on the ambient ring $T$,  and if this ring is not clear from the context, we will write 
$M(k, \divd)=M_T(k,\divd)$ and $J(\divd^\bullet) =J_T(\divd^\bullet)$.

 \end{definition} 
 
 The $M(k, \divd)$ can be easily described in terms of the functors developed in the previous section, as follows.

\begin{lemma}
\label{lem:Mmult1}
Let $\divd$ be an effective divisor on $E$, and let $G_{\divd}$ be the functor defined in Notation~\ref{not:Gd}.  
If $M(k,\divd)$ is defined as above, then its image $\pi(M(k,\divd)) \in X$ is equal to 
\[
[ \tau^{-k+1}(G_\divd) \circ \tau^{-k+2}(G_\divd) \circ \dots \circ \tau^{-1}(G_\divd) \circ G_\divd](\mc{O}_X).
\]
\end{lemma}
\begin{proof}
This follows directly from Lemma~\ref{lem-Fcomp2}.  
\end{proof}

We want to study the properties of the right ideals $M(k, \divd)$, and 
our first goal is to   compute their Hilbert series precisely. 
We start with a routine lemma. 

\begin{lemma}\label{lem:nondecreasing}
If $J$ is a saturated right ideal of $T$ such that $M = T/J$ is $g$-torsion with $\GKdim M = 1$, then 
$M$ has a filtration  whose factors   are shifted point modules.  Also
$\dim_\kk M_n \leq \dim_\kk M_{n+1}$ for all $n \in \ZZ$.
\end{lemma}
\begin{proof}   
As $M$ has no finite-dimensional submodules, it has a finite chain of submodules, each factor of which is 1-critical (see \cite[(6.2.19) and 
(6.8.25)]{MR}) and by induction it suffices to prove the result when $M$ is $1$-critical.
   As noted in \cite[Lemma~3.8]{RSS2},  $M$ is then  killed by $g$ and so $M$ is a module over 
$T/gT \cong B(E, \mc{M}, \tau)$.  
Since $T/gT$ is generated in degree $1$ (see Hypothesis~\ref{hyp:main}), the result follows from  Lemma~\ref{lem:pt-crit}(3).  
\end{proof}

\begin{lemma}
\label{hs-lem}  
  Let $\divd^\bullet = (\divd^0, \dots, \divd^{k-1})$ be an allowable divisor layering.
  Set $J = J( \divd^\bullet) $ and $s =  \sum_i \deg \divd^i$.
\begin{enumerate}
\item $\dim_\kk (T/J)_n \leq s$ for all $n \geq 0$, with equality for $n \gg 0$.

\item Let $L = J(\dive^\bullet)$ where $\dive^\bullet = (\tau(\divd^1), \tau(\divd^2), \dots, \tau(\divd^{k-1}))$.  Then 
$J \cap Tg = Lg$ and 
\begin{equation}\label{hs-equ}
h_{T/J}(t) = h_{\overline{T}/\overline{J}}(t) + t \, h_{T/L}(t).
\end{equation}
\item  Suppose that $\deg \divd^i < \mu (\ell -i)$ for all $0 \leq i \leq k-1$, and some $\ell \geq k$. 
Then $\dim_\kk (T/J)_n = s$ for all $n \geq \ell$ and 
$\overline{J}_n = \bigoplus_{n \geq 0} H^0(E, \mc{M}_n(-\divd^0))$ for $n \geq \ell$. \end{enumerate}
\end{lemma}

\begin{proof}
(1)  By Lemma~\ref{lem:nondecreasing}, $M$ has a filtration whose factors which are shifted point modules.   By Lemma~\ref{lem:layer}, the sum 
of all of the $\divd^i$ counts with multiplicity the number of shifted point modules which occur in this filtration.
Thus   $\dim_\kk (T/J)_n = s$ for $n \gg 0$ and so,  by Lemma~\ref{lem:nondecreasing}, 
$\dim_\kk (T/J)_n \leq s$ for all $n \geq 0$.

(2) Define $L$ by  $Lg = J \cap Tg$. Thus  $L$ is a saturated right ideal of $T$ and we first need to prove 
 that $L = J( \dive^\bullet)$ for the given $\dive^\bullet$. To see this,   reduce to the case that the divisors $\divd^i$ are all supported on a single orbit $\mb{O}(p)$. 
  In $X$, we have
 $\pi(Lg) = \pi(J) \cap \pi(Tg)$.
By the proof of Lemma~\ref{lem:lattice}, the functor $( \wh{-} )_p$ of Theorem~\ref{thm:vdb} 
preserves intersections.   Applying this functor to our equation and using 
Theorem~\ref{thm:vdb}(2) gives 
$\mf{L} N = \mf{J}(\divd^\bullet) \cap N C_p$, where $\mf{L} = (\wh{\pi(L)})_p$.  
By definition, $\mf{J}(\divd^\bullet) \cap N C_p$ has 
$0$ in each $(k,k)$-spot, and $(x)^{a_{k,\ell}}$ in the $(k, \ell)$-spots with $k > \ell$, where 
$\divd^i = \sum_j a_{j+i, j} \tau^j(p)$.  On the other hand,    
\[\dive^i = \tau(\divd^{i+1}) = \sum_j a_{j+i+1,j} \tau^{j+1}(p) = \sum_j a_{j+i, j-1} \tau^j(p).\]
Then $\mf{J}(\dive^\bullet)$ has $(x)^{a_{k, \ell-1}}$ in the $(k, \ell)$-spot, and so a calculation shows that 
$\mf{J}(\dive^\bullet) N$ has $(x)^{a_{k,\ell}}$ in the $(k, \ell)$-spot for $k > \ell$ and $0$ in each $(k, k)$-spot.  Hence 
$\mf{J}(\divd^\bullet) \cap N C_p = \mf{J}(\dive^\bullet) N$.  Thus $\mf{L} = \mf{J}(\dive^\bullet)$.   
So $\pi(L) = \pi(J(\dive^{\bullet}))$, and taking saturations gives $L = J(\dive^\bullet)$,  as claimed.

In order to compute \eqref{hs-equ}, use the exact sequence   
$
0 \to Tg/(J \cap Tg)   \to T/J \to T/(J + Tg)   \to 0, 
$
noting that  $Tg/(J \cap Tg) = Tg/Lg \cong (T/L)[-1]$ and $T/(J + Tg) \cong \overline{T}/\overline{J} $.

(3) By Part (1), to show $\dim_\kk (T/J)_n = s$ for all $n \geq \ell$, it is enough to show that 
$\dim_{\kk} (T/J)_n \geq s$ if $n \geq \ell$.
By Lemma~\ref{lem:layer}(2),  
$\overline{J}^{sat} = \bigoplus_{n \geq 0} H^0(E, \mc{M}_n(-\divd^0))$.
By assumption $\deg \divd^0 < (\deg \mc{M}) \ell \leq \deg \mc{M}_n$ for $n \geq \ell$, 
and so the divisor $\divd^0$ 
presents independent vanishing conditions to sections of the sheaf $\mc{M}_n$ 
for $n \geq \ell$; in other words,  $\dim_{\kk} H^0(E, \mc{M}_n(-\divd^0)) = \dim_{\kk} H^0(E, \mc{M}_n) - \deg \divd^0$.
Thus $\dim_{\kk} (\overline{T}/\overline{J})_n \geq \deg \divd^0$ for $n \geq \ell$.

Let $Lg = J \cap Tg$ as in Part (2).  Then $L = J(\dive^\bullet)$ in the notation of that part and hence $L$ has fewer divisor layers; the divisor 
layering $\dive^\bullet$ also satisfies $\deg \dive^i = \deg \divd^{i+1} < (\deg \mc{M})(\ell - i - 1)$ for $0 \leq i \leq k -2$.
By induction we have $\dim_{\kk} (T/L)_n = \sum_{i=0}^{k-2} \deg \dive^i = \sum_{i=1}^{k-1} \deg \divd^i$ for $n \geq \ell - 1$.   
Now by Part (2), we see that $\dim_{\kk} (T/J)_n \geq s$ for $n \geq \ell$, and hence 
$\dim_{\kk} (T/J)_n = s$ for $n \geq \ell$.  Examining the proof above, this also shows that the inequality
$\dim_{\kk} (\overline{T}/\overline{J})_n \geq \deg \divd^0$ for $n \geq \ell$ must have been an equality for each $n \geq \ell$, 
or equivalently that $\overline{J}$ must have been saturated in degrees $\geq \ell$, proving the second part.
\end{proof}

When applied to the special case of the right ideals $M(k, \divd)$ of $T$ as defined in Definition~\ref{def:M(k,D)}, 
 Lemma~\ref{hs-lem} gives the following result.
\begin{proposition}\label{prop-M}
Let $\divd$ be an effective divisor on $E$ with $\deg \divd < \deg \mc{M}$.
Then for all $0 \leq k \leq n$, we have:
\begin{enumerate}
 \item $\dim M(k,\divd)_n = \dim T_n - (\deg \divd) \binom{k+1}{2}$.
\item $(M(k, \divd)\cap gT)_n = g M(k-1, \divd)_{n-1}$.
\item $\bbar{M(k, \divd)}_n = H^0(E, \sM_n(-\divd_k))$.  \qed
\end{enumerate}
\end{proposition}

\begin{remark}
\label{rem:neg}  
Using the  convention that $M(k, \divd) = T$ if $k \leq 0$, Proposition~\ref{prop-M}(2) 
is easily seen to  hold for all $k \leq n \in \mb{Z}$.
\end{remark}

Using the functors developed in the Section~\ref{FUNCTORIAL}, we next show that a product of two graded pieces of 
right ideals of the form  $J(\divd^\bullet)$ is sometimes again a graded piece of some $J(\divc^\bullet)$.  

Recall that a graded vector subspace $V \subseteq T$ is \emph{$g$-divisible} if $V \cap Tg = Vg$.  
Given graded modules $M,N\subseteq P$ over a cg algebra $A$, we write $M\ppe N$\label{ppe-defn}  if $M$ and $N$ agree 
up to a finite dimensional   vector space. Finally, recall that $\omega: \rqgr A\to \rGr A$ denotes the section functor.

\begin{lemma}
\label{lem-compose}
Let $X = \rqgr T$.
\begin{enumerate}
\item
Let $\Sigma = [1]$ be the shift functor on $X$. 
 Then $\Sigma \circ F_p = F_{\tau(p)} \circ \Sigma$.
\item
Let $H$ be any finite composition of functors of the form $F_q$ for   $q \in E$. 
Then 
\[
\omega(\mc{G})_m \cdot \omega(H(\mc{O}_X))_n \subseteq \omega([\tau^{-m}(H)](\mc{G}))_{n+m}
\]
 for any noetherian object $\mc{G} \in X$ and any $m,n \geq 1$.
 \end{enumerate}
\end{lemma}

\begin{proof}
(1) This is equivalent to  the following 
statement:  If $P(p)$ is the point module of $T$ corresponding to the point $p \in E$, as defined before
  Lemma~\ref{lem:pt-crit}, then $P(p)[1] \ppe P(\tau(p))$.  This follows from a direct calculation.

 (2) Set $ \wt{T} = \omega(\pi(T))$. In   Lemma~\ref{lem:T-props} we noted that $\wt{T}=T$,
  but here we need the explicit    multiplication rule for $\wt{T}$. 
 From  \cite[Section~4]{AZ} $ \wt{T}  = \bigoplus_{n \geq 0} \Hom_{X}(\mc{O}_X, \mc{O}_X[n]),$
 with its multiplication  $\star$ defined by 
$f \star g = \Sigma^n(f) \circ g \in \Hom_X(\mc{O}_X, \Sigma^{m+n}(\mc{O}_X))$
for $f \in \wt{T}_m$ and $ g \in \wt{T}_n$.

Let
$ 
f \in \omega(\mc{G})_m = \Hom_X(\mc{O}_X, \Sigma^m(\mc{G})),\  $ and  
$g \in \omega(H(\mc{O}_X))_n = \Hom_X(\mc{O}_X, \Sigma^n(H(\mc{O}_X))) \subseteq T_n.
$  Then,  using the functoriality of $H$, one sees that $f\star g \in \omega(\mc{G})_{n+m}$ is the element
\[
\xymatrix{
\sO_X \ar[r]^(0.4){g} & \Sigma^n(H(\mc{O}_X)) \ar[rr]^(0.48){\Sigma^n(H(f))} && [\Sigma^n \circ H \circ \Sigma^m](\mc{G}).}
\]
 Since $H \circ \Sigma^m = \Sigma^m \circ \tau^{-m}(H) $ by Part~(1), 
\[
f \star g \in \Hom_X(\mc{O}_X, [\Sigma^{n+m} \circ \tau^{-m}(H)](\mc{G})) = \omega(\tau^{-m}(H)(\mc{G}))_{n+m},
\]
as required.
\end{proof}

In particular, we can apply the result above to the right ideals $M(k, \divd)$ of $T$.   

\begin{lemma}
\label{lem:Mmult}
Let $\divd$ be an effective divisor on $E$.
%, and let $G_\divd$ be the functor defined in Notation~\ref{not:Gd}.  
\begin{enumerate}
\item Let $n, m, k, \ell \in \NN$.  Then
\begin{equation}\label{equ:Mmult}
 M(k, \divd)_m M(\ell, \tau^{m-k}(\divd))_n \subseteq M(k + \ell, \divd)_{m+n}.   
\end{equation}
\item  If $\deg \mc{M} - \deg \divd\geq 2$,  then \eqref{equ:Mmult}  is an  equality   for all $k, \ell \geq 0$ and all $m \geq k, n \geq \ell$.
\break
If  $\deg \mc{M} - \deg \divd = 1$, then \eqref{equ:Mmult}  is an  equality   for all 
$k, \ell \geq 0$ and  $m\geq \max\{2, k\}$ and $n\geq \max\{2, \ell\}$.
\end{enumerate}
\end{lemma}

\begin{proof}
(1)   Let $G_\divd$ be the functor defined in Notation~\ref{not:Gd}.
Note that $\mc{G} = \pi(M(k, \divd)) = G'(\mc{O}_X)$, where 
$G' = \tau^{-k+1}(G_\divd) \circ \tau^{-k+2}(G_\divd) \circ \dots \circ G_\divd$, by Lemma~\ref{lem:Mmult1}.
Similarly, defining  $H = \tau^{m-k-\ell+1}(G_\divd) \circ \dots \circ \tau^{m-k}(G_\divd)$, then  
$H(\mc{O}_X) = \pi(M(\ell, \tau^{m-k}(\divd)))$.
  
Now apply Lemma~\ref{lem-compose}(2) to $\mc{G}$ and $H$.
Since 
$
\tau^{-m}(H) \circ G' = \tau^{-k-\ell+1}(G_\divd) \circ \dots \circ G_\divd,
$
the result follows.

(2) Consider the inclusion \eqref{equ:Mmult}. If we prove for some values of $k, \ell, m, n$ both that  
\begin{equation}
\label{need1}
\overline{M(k, \divd)}_{m} \overline{M(\ell, \tau^{m-k}(\divd))}_n = \overline{M(k + \ell, \divd)}_{m+n}
\end{equation}
and that 
\begin{equation}
\label{need2}
\bigl(M(k, \divd)_{m} M(\ell, \tau^{m-k}(\divd))_n\bigr) \cap g T =  M(k + \ell, \divd)_{m+n} \cap gT,
\end{equation}
then   both sides of \eqref{equ:Mmult} will have the same vector space dimension and so the 
inclusion will be an equality.

We start with \eqref{need2}. Recall from Remark~\ref{rem:neg} that 
Proposition~\ref{prop-M}(2) actually holds for all $k  \in \ZZ$ with $k \leq n$.
We now allow negative $k$ and $\ell$ in equation \eqref{need2}, though we continue to assume that $m,n \geq 0$. 
This will allow us to prove \eqref{need2} by induction on $m + n$.  The base case $m = n = 0$ is trivial.
If $m \geq 1$, then 
\begin{gather*}
M(k,\divd)_m M(\ell, \tau^{m-k}(\divd))_n \cap g T  \ \supseteq  \ g M(k-1, \divd)_{m-1} M(\ell, \tau^{m-k}(\divd))_n \\
= \
gM(k+ \ell -1, \divd)_{m+n-1} \ =\ M(k + \ell, \divd)_{m+n} \cap g T,
\end{gather*}
where we have used Proposition~\ref{prop-M}(2) and the induction hypothesis.   The case $n \geq 1$ is similar.

Thus it remains to consider \eqref{need1}.  By Proposition~\ref{prop-M}(3), 
\begin{gather*}
\overline{M(k, \divd)}_m = H^0(E, \mc{M}_m(-\divd_k)), \ \ \ \overline{M(\ell, \tau^{m-k}(\divd))}_n = 
H^0\left(E, \mc{M}_n\left(-\tau^{m-k}(\divd_{\ell})\right)\right), \\ \text{and}\ \  \overline{M(k + \ell, \divd)}_{n+m} = H^0(E, \mc{M}_{n+m}(-\divd_{k + \ell})).  
\end{gather*}
Set $\mc{N} = \mc{M}_m(-\divd_k)$ and $\mc{P} = \mc{M}_n(-\tau^{m-k}(\divd_{\ell}))$.
By the definition of multiplication in $B(E, \mc{M}, \tau)$ the goal is to show that the natural multiplication map 
\[
H^0(E, \mc{N}) \otimes H^0(E, \mc{P}^{\tau^m}) \to   H^0(E, \mc{N} \otimes \mc{P}^{\tau^m})
\]
is surjective.
Recall from Lemma~\ref{lem:sec-mult} that this is always true provided 
 \begin{equation}\label{need3}
 2\leq \deg \mc{N} = m\deg \sM - k \deg \divd \quad\text{ and }\quad  
2\leq \deg \mc{P}^{\tau^m}  = n\deg \mc{M} -\ell \deg \divd,
\end{equation} except that if $\deg \mc{N} = \deg \mc{P}^{\tau^m}  = 2$ 
one 
requires the additional condition that  $\mc{N} \not \cong \mc{P}^{\tau^m}$.
Since \eqref{need3} always holds,  it remains to check that this final assertion is also true. 
This  routine calculation is left to the reader:  although there are a number of special cases, they all reduce to the 
fact that $\mc{N}\not\cong\mc{N}^{\tau^m}$ for $m\geq 1$.
\end{proof}

The next result, which will be used in \cite{RSS2}, illustrates the usefulness of Lemma~\ref{lem:Mmult}.

\begin{corollary} \label{cor:multcomp}
 Let $\divd$ be an effective divisor on $E$ of degree $\leq \mu-2$, and let $k \in \NN$.    
Then $$T_1 M(k, \divd)_k = M(k, \tau^{-1}(\divd))_k T_1.$$
\end{corollary}
\begin{proof}
 By Lemma~\ref{lem:Mmult}(2),  
\[ M(0, \tau^{-1} (\divd))_1 M(k, \divd)_k = M(k, \tau^{-1}(\divd))_{k+1}
 = M(k, \tau^{-1}(\divd))_k M(0, \tau^{-1}(\divd))_1. \qedhere
\]
\end{proof}

%%%%%%%%%%%%%%%%%%%%%%%

\section{The blowups  \texorpdfstring{$T(\divd)$}{LG} }\label{BLOWUP}

 Throughout this section $T$ will satisfy Hypothesis~\ref{hyp:main}.
Here we define algebras $T(\divd)$ that correspond to blowing up $T$ at the points on an effective divisor $\divd$, 
and give a uniform proof of their properties.  

\begin{definition}
\label{def-blowup} 
For any effective divisor $\divd$ on $E$ with $\deg \divd < \mu$,  define a subalgebra $T(\divd)$ of $T$ by putting $T(\divd) = \bigoplus_{k \geq 0} M(k,\divd)_k \subseteq T$.  
More generally, for any $\ell \geq 0$
define a graded $T(\divd)$-submodule of $T$ by   
\[ T_{\leq \ell}*T(\divd) := \bigoplus_{n \geq 0} M(n-\ell, \divd^{\tau^\ell})_n.\]

The facts that $T(\divd)$ is a subalgebra of $T$, and then 
that   $T_{\leq \ell}*T(\divd)$ is a right $T(\divd)$-module, follow immediately from 
Part (1) of Lemma~\ref{lem:Mmult}.

The notation $T_{\leq \ell} * T(\divd)$ is meant to recall multiplication, since this module  
  equals  
$T_{\leq \ell} T(\divd)$ except when $\deg \divd$ is large.  We will prove this and other
 properties in the next few results.
 \end{definition} 

 One should think of $T(\divd)$ as a ring-theoretic blowup of $T$ along the divisor $\divd$. This is justified  in part by 
the properties proven in the next few sections (see Proposition~\ref{prop:subtle}  and Remark~\ref{line-defn}
 among others). The analogy is particularly strong when  $\deg \divd<\mu-1$,
so $T(\divd)$ is generated in degree one, as the algebra then corresponds to the (commutative) definition of a blowup 
in terms of Rees rings.   It  is still 
 appropriate for general $\divd$ since one can show that $T(\divd)$ is the same as the  section ring of the blowup defined by Van den Bergh
 \cite{VdB} (see, in particular,   the del Pezzo 
 surfaces considered   in \cite[Section~11]{VdB}).  
 A further discussion on this analogy can also be found in~\cite{Rog09}.

\begin{proposition}\label{elempropsofR}
Let $\divd$ be an effective divisor on $E$ with $d= \deg \divd < \mu$.   Let $n, \ell \in \NN$.  Then 
\begin{enumerate}
 \item $\dim_\kk (T_{\leq \ell}*T(\divd))_n = \dim_\kk T_n - d\binom{n-\ell+1}{2}$.   
\item $(T_{\leq \ell}*T(\divd) \cap gT)_n = g (T_{\leq \ell}*T(\divd))_{n-1}$; that is, $T_{\leq \ell}*T(\divd)$ is $g$-divisible.
\item $\bbar{T_{\leq \ell}* T(\divd)}_n  =H^0(E, \sM_n(-\divd^{\tau^{\ell}} - \dots - \divd^{\tau^{n-1}}))$.
\end{enumerate}
\end{proposition}
\begin{proof}
 This is immediate from  Proposition~\ref{prop-M}.
\end{proof}

\begin{theorem}
\label{thm:Rnoeth} 
Let $T$ satisfy Hypothesis~\ref{hyp:main}, and 
let $\divd$ be an effective divisor on $E$ with $d=\deg \divd < \mu$.   Let $R = T(\divd)$.
Then:
\begin{enumerate}
\item $R$ has Hilbert series $h_{R}(t) = h_{T}(t) - d/(1-t)^3$.
\item $R$ is $g$-divisible with $R/gR = B(E, \mc{M}(-\divd), \tau)$.
\item $R$ is strongly noetherian, Auslander-Gorenstein, CM, and a maximal order.  Further, $R$ satisfies the 
Artin-Zhang $\chi$ conditions, and has cohomological dimension 2 and   a balanced dualizing complex.
 \item $R$ is generated in degree $1$ if $\mu - d \geq 2$.  If $\mu - d = 1$, then $R_m R_n = R_{m+n}$ for all $m,n \geq 2$;
  in particular, $R$ is generated in degrees $\leq 2$, and 
the Veronese ring $R^{(2)}$ is generated in degree $1$.
\item  $T_{\leq \ell}*T(\divd)$ is a finitely generated right $R = T(\divd)$-module for each $\ell \geq 0$.
\item  If $\mu - d \geq 2$ then $T_{\leq \ell}* T(\divd) = T_{\leq \ell} T(\divd)$.
\end{enumerate}
\end{theorem}

\begin{proof} This summarises earlier results. In more detail, 
Parts~(1,2) follow  from Proposition~\ref{elempropsofR} with $\ell=0$;
Part (3) is a  consequence of Proposition~\ref{prop:foo}; while Parts (4,5) 
are consequences of Proposition~\ref{lem:Mmult}(2).
 
   Since both sides of Part~(6) are equal to $T_n$ in degree $n \leq \ell$, it is enough to show that 
   $T_{\ell} M(n-\ell, \divd)_{n- \ell} = M(n-\ell, \tau^{-\ell}(\divd))_n$ for all $n \geq \ell$.  
Because $\deg \divd \leq \mu -2$, this also follows from Proposition~\ref{lem:Mmult}(2).  
\end{proof}

 In \cite{VdB} Van den Bergh only defines blowing up at a single point, and so it is useful to know that our
 multiple blowups can also be defined iteratively.
 
\begin{proposition}
\label{prop:iterate}
Let $\divd$ be a divisor on $E$ with $\deg \divd < \mu$.  Suppose that 
$\divd = \divc + \dive$, where $\divc$, $\dive$ are effective, and let ${R} = T(\divc)$.  Then $T(\divd) = {R}(\dive) = (T(\divc))(\dive)$.
\end{proposition}

Before proving this, we give a lemma, which will be needed to handle the case $\deg \divd = \mu -1$.

\begin{lemma}\label{lem:1point}
Suppose that $\deg \divc = \mu-2$ and write ${R}=T(\divc)$.     
Let $q \in E$.  Define  $V = \bigoplus V_n \subset T$ by
\[ V_n = J_T\left(\divc_n +q, \tau^{-1}(\divc_{n-1}), \dots, \tau^{-n+1}(\divc)\right)_n.\]
Then $V = J_{ {R}}(q)$.
\end{lemma}
\begin{proof}
It follows by comparing divisor data that $V \subseteq {R}$.  By Lemma~\ref{lem:layer}(3),
\[ J_{{R}}(q)_n = \left\{ x \in {R}_n \st \bbar{x} \in H^0(E, \sM(-\divc)_n (-q))\right\}.\]
The top-layer vanishing of $V_n$ gives that  $V \subseteq J_{{R}}(q)$.  
By Lemma~\ref{hs-lem}(3), $V$ and $J_{{R}}(q)$ have the same Hilbert series, and thus the two are equal.
\end{proof}

\begin{proof}[Proof of Proposition~\ref{prop:iterate}]
In most cases, the proposition follows  from 
 the fact that   the rings in question are generated in degree $1$.  
Specifically, suppose that $\deg \divd \leq \mu -2$.  Then,   by 
Theorem~\ref{thm:Rnoeth}, $T(\divd)$ is generated in degree 1.  
Similarly, as  $\deg \dive \leq \deg \mc{M}(-\divc)  -2 = \mu - \deg \divc -2$ the same argument implies
 that ${R}(\dive)$ is generated in degree $1$.  But now 
$T(\divd)_1 = \{ x \in T_1 \st \overline{x} \in H^0(E, \mc{M}(-\divd)) \}$ and,
 similarly, ${R}_1 = T(\divc)_1 = \{ x \in T_1 \st \overline{x} \in H^0(E, \mc{M}(-\divc)) \}$.  
 Then ${R}(\dive)_1 = \{y \in {R}_1 \st \overline{y} \in H^0(E, \mc{M}(-\divc)(-\dive)) \} = T(\divd)_1$
 and hence ${R}(\dive)=T(\divd)$.

Thus the only issue is when $\deg \divd = \mu - 1$,  so   assume that  this is the case.  
A routine induction reduces to the case when $\deg \dive=1$, so assume that $\dive=p$ is a closed point on $E$.
%and set $R = \wt{R}(p)$. 
By Theorem~\ref{thm:Rnoeth}, the rings $T(\divd)$ and ${R}(p)$ are generated in degrees $1$ and $2$.  They are equal in degree 
$1$ by the argument from  the first paragraph of the proof and so, in order to prove the proposition,  it suffices 
to show that they are equal in degree $2$. 

Define a graded subspace $J = \bigoplus_{n \geq 0 } J_n$ of $T$ by setting 
\[J_n=  J_T(\divc_n + p + \tau^{-1} (p),  \  \tau^{-1}(\divc_{n-1} + p), \ \tau^{-2}(\divc_{n-2}), \dots, \tau^{-n+1}(\divc_1))_n,\]
in the notation of Definition~\ref{def:M(k,D)}.  It is immediate from the divisor data for the $J_n$ that $J \subseteq {R}$.  
Also, $J_2 = T(\divd)_2$.  
%By Lemma~\ref{hs-lem}, we have $\dim_\kk J_n = \dim_\kk {R}_n - 3$ for $n \geq 2$.
As in the proof of Lemma~\ref{lem:Mmult}(1),  it follows from Lemmas~\ref{lem-compose}(2) and~\ref{lem-Fcomp2} that 
$J$ is a right ideal of ${R}$.
Let $J' = M_{{R}}(2, p) = J_{{R}}(p + \tau^{-1}(p), \,  \tau^{-1}(p))$ which, by definition,  is a saturated right ideal of ${R}$  
with $J'_2 = {R}(p)_2$.  
Clearly $\tau^{k-1}(p) \not\in \tau^{-2}(\divc)$ for $k\gg 0$, so choose some such $k\geq 2$.  We next show that $J'_k = J_k$.  

We establish some notation.  
Let $V = J_T(\tau^{k-1}(p) + \divc_2,\,  \tau^{-1}(\divc))_2$
and $I = J_{{R}}(p)$.
We now work in $X=\rqgr T$, so let $\omega: \rQgr T \to \rGr T$ be the right adjoint to $\pi:  \rGr T \to \rQgr T$.  
For  an effective divisor $\divx$ on $E$, let $G_{\divx}$ be the functor on $\rqgr T$ defined in Notation~\ref{not:Gd}.   
Define $\sI, \sK \in X $ by
 \begin{align*}  \sI & \ =\  \sJ_T(p + \divc_k, \tau^{-1}(\divc_{k-1}), \dots, \tau^{-k+1}(\divc)), \quad\text{and}\\
  \sK & \ =\ G_{\tau^{k-1}(p)} \circ G_{\tau^{-1}(\divc)} \circ G_{\divc}(\sO_X) \ =\  G_{\tau^{k-1}(p)} \sJ_T(\divc_2, \tau^{-1}(\divc)),
  \end{align*} 
where the final equality   follows from Lemma~\ref{lem-Fcomp2}.
By Lemma~\ref{lem:1point}, we have  $V = J_{{R}}(\tau^{k-1}(p))_2$ and $  ( \omega \sI)_k = I_k.$
We compute divisor data on $T$ for $\sK$.  By Lemma~\ref{lem-Fcomp}, $\sK = \sJ_T(\divx^0, \divx^1, \divx^2)$, where:
\begin{align*}
\divx^0 & =   \tau^{k-1}(p) + \divc+ \tau^{-1}(\divc) = \tau^{k-1}(p) + \divc_2, \\
\divx^1 & = \min(\tau^{k-1}(p) + \tau^{-1}(\divc), \,  \tau^{-1}(\divc)+ \tau^{-2}(\divc)), \\
\divx^2 & = \min(\tau^{k-1}(p), \tau^{-2}(\divc)).
\end{align*} 
Now, $\divx^2 = 0$ by choice of $k$, and  $\divx^{\bullet} = (\divc_2 + \tau^{k-1}(p),\,  \tau^{-1}(\divc))$.  
Thus $V = (\omega \sK)_2$.  

Apply Lemma~\ref{lem-compose}(2) to the product $I_k V = (\omega \sI)_k (\omega \sK)_2$
to get
\begin{equation}\label{unique0}
 I_k V \ \subseteq \  W\ = \  \omega([\tau^{-k}(G_{\tau^{k-1}(p)} \circ G_{\tau^{-1}(\divc)} \circ G_{\divc})](\sI))_{k+2}.
 \end{equation}
We use previous results to compute $W$.  
Let $\divw^{\bullet}$ be the defining data for $\tau^{-k}( G_{\tau^{-1}(\divc)} \circ G_{\divc})(\sI)$.
By Lemma~\ref{lem-Fcomp2}, $\divw^0 =  p + \divc_{k+2}$.  
The formula for  $\divw^i$ for $i \geq 1$ is a bit more complicated.  However, by Lemma~\ref{lem-Fcomp2}, $\divw^i \leq \divz^i + \tau^{-k}(\divc)+\tau^{-k-1}(\divc)$, where $\divz^\bullet$ defines $\sI$.  
From the formula for $\sI $, we see that 
$\divw^i \leq \tau^{-i}(\divc_{k+2-i})$.
On the other hand, 
\[\tau^{-k}(G_{\tau^{-1}(\divc)} \circ G_{\divc})(\sI )\subseteq G_{\tau^{-k-1}(\divc)} \circ G_{\tau^{-k}(\divc)} \sJ_T(\divc_k, \tau^{-1}(\divc_{k-1}), \dots),\]
and so $\divw^i \geq \tau^{-i}(\divc_{k+2-i})$, by Lemma~\ref{lem-Fcomp2} once again.  Thus $\divw^i =\tau^{-i}( \divc_{k+2-i})$ for $i \geq 1$.

By Lemma~\ref{lem-Fcomp}, $\tau^{-k}(G_{\tau^{k-1}(p)} \circ G_{\tau^{-1}(\divc)} \circ G_{\divc})(\sI)$ is then defined by $\divx^{\bullet}$ where
\[ \divx^0 =  p + \tau^{-1}(p) + \divc_{k+2}, \quad \divx^1 = \tau^{-1}(p) + \tau^{-1}(\divc_{k+1}),\]
and for $i \geq 2$ we have
\[ \divx^i = \min\bigl(\tau^{-i}(\divc_{k+2-i})+ \tau^{-1}(p),\  \tau^{-i}(\divc_{k+3-i})\bigr) \geq  \tau^{-i}(\divc_{k+2-i}).\]
Thus  
\[
W \subseteq J_T(p+ \tau^{-1}(p) + \divc_{k+2}, \ \tau^{-1}(p) + \tau^{-1}(\divc_{k+1}),\  \tau^{-2}(\divc_k), \dots )_{k+2} = J_{k+2}.
\]
By Lemma~\ref{lem:Mmult}(2), on the other hand, we have
\begin{equation}\label{unique1}
 I_k V  = J_{{R}}(p)_k J_{{R}}(\tau^{k-1}(p))_2 = M_{{R}}(1,p)_k M_{{R}}(1 ,\tau^{k-1}(p))_2 = M_{{R}}(2, p)_{k+2} = J'_{k+2} .
 \end{equation}
Thus $J'_{k+2} \subseteq J_{k+2}$ by \eqref{unique0}. However, by Lemma~\ref{hs-lem}(3),  twice, $\dim_\kk J_n = \dim_\kk {R}_n - 3 = \dim_{\kk}J'_n$ for $n \geq 2$. Thus  $J'_{k+2}= J_{k+2}$  for all $k \gg 0$.

Finally, let $\omega_{{R}}$ and $\pi_{ R}$ denote the relevant functors on the module categories of ${R}$. As $J'$ is a saturated right ideal of ${R}$, we have 
$ J' =  \omega_{ R} \pi_{ R}(J') = \omega_{ R} \pi_{ R}(J) \supseteq J.$
Comparing Hilbert series again, we see that $J'_n=J_n$ in degrees $n\geq 2$.  In particular, 
\begin{equation}\label{unique2}
T(\divd)_2 = J_2 = J'_2 = {R}(p)_2,
\end{equation}
which is what we needed to prove.
\end{proof}

\begin{remark}\label{uniqueness-remark}
 A subtle but important point is that, a priori, the construction of $J(\divd^\bullet)$ and hence $T(\divd)$  depends upon a choice of the
  functors $(\widehat{-})_p$ in 
 Theorem~\ref{thm:vdb}; these functors are not unique, and different functors may give different $J$'s.  
 However, the   proof of Proposition~\ref{prop:iterate} shows  that $T(\divd)$ is   uniquely determined by $\divd$. To see this
  note  that, by Lemma~\ref{lem:layer},
  $M(1,\divd) = \bigoplus \{ x \in T_n \st \overline{x} \in H^0(E, \mc{M}_n(-\divd)) \}$  and hence $T(\divd)_1$  is  clearly determined by $\divd$. 
  If $\deg \divd < \mu -1$ then $T(\divd)$ is generated in degree 1 and so is determined by $\divd$.  So suppose that 
  $\deg \divd = \mu -1$, and write $\divd = \divc + p$ as in the above proof. Then \eqref{unique1}    
  shows   that $M_{{R}}(2, p)_{k+2}$ for $k\gg 0$     is determined by the appropriate modules
 $M(1,\tau^j(p))$.   Hence the saturation $M_{{R}}(2, p)$, and thus by \eqref{unique2} 
$T(\divd)_2 = T(\divc)( p)_2 = M_{{R}}(2, p)_2$, is determined by $  \divc$ and $ p$.
This last equation holds for any choice of $\divc$ and $p$, and so $T(\divd)_2$ is  determined  by $\divd$.
  Since $T(\divd)$  is generated in degrees 1 and 2, it is uniquely determined by $\divd$.

   This is also useful in a second case, when one takes an algebraically closed field extension $\kk\subset K$.
    In this case, given  a divisor  $\divd$  on $ E_\kk$ we can either form $T(\divd)\otimes_{\kk}K$ or regard $\divd $ as a 
    divisor on $E\otimes_\kk K$ and form $(T\otimes_\kk K)(\divd)$.
   The potential difference between these two constructions lies in the fact that one is using two distinct versions of $(\widehat{-})_p$: 
   either defined directly on 
$ C_p(K)$ or by taking the functors from $C_p(\kk)$ and  then tensoring up with $K$. 
It is easy to see that $M_{T\otimes K}(1, \divd) = M_T(1, \divd) \otimes_{\kk} K$, and a similar argument to the previous paragraph shows 
that $T(\divd)\otimes_{\kk}K= (T\otimes_\kk K)(\divd)$.\end{remark}

To end the section, we use the earlier results to construct right $T(\divd)$-modules corresponding to certain very particular kinds of divisor data.  
These will be used in the companion paper \cite{RSS2} to give explicit examples 
of interesting endomorphism rings of modules over $T(\divd)$, known as  \emph{virtual blowups}.

For a divisor $\divz\in \Div E$, write  $\divz = \sum z_q q$ and  let 
$\divz_+ = \max(\divz, 0) = \sum_{\{q \st z_q \geq 0\}} z_q q$ be the minimal effective divisor greater than or equal to  $\divz$.

\begin{lemma}\label{lem:buildM}
 Let $\divd$ be an effective divisor on $E$  
 so that $\deg \divd <\mu$.  Let $\divy$ be a divisor on $E$ so that    $0 \leq \divy \leq \divd_k$  for some $k$.  Then there is a $g$-divisible
  finitely generated right $T(\divd)$-module $M$ with $T(\divd) \subseteq M \subseteq T$ so that 
  $\bbar{M} \ppe \bigoplus_n H^0(E, \sM_n(-\divd_n+\divy))$. 
 \end{lemma}

\begin{proof}
 If $ 0 \leq i \leq m-1$, let
$ \divz^{m,i} = \tau^{-i} (\divd_{m-i}-\divy)_+.$
Note that, if $\divz \leq \divz' \in \Div(E)$, then $\divz_+ \leq \divz'_+$. Thus, 
 for fixed $m$, it follows   that $\divz^{m,\bullet}$ is an allowable divisor sequence.
For $m \in \NN$, define
$M_m = J(\divz^{m,\bullet})_m$; thus $M=\bigoplus_{m\geq 0}M_m \subseteq T$.
Since $\divz^{m+1,i+1} = \tau^{-1}(\divz^{m,i})$, using Lemma~\ref{hs-lem}(2) one graded piece at a time we see that 
$M \cap Tg = Mg$.  It is easy to see that $M \supseteq T(\divd)$ by comparing the divisor sequence
$\divz^{m,\bullet}$ for $M_m$ with the divisor sequence $(\tau^{-i}(\divd_{m-i}))_i$ for $T(\divd)_m$.

It remains to  show that $M$ is a right $T(\divd)$-module, for which we use  Lemma~\ref{lem-compose}(2).  
Fix $ n,m \in \ZZ_{\geq 1}$. 
Let $\sF:= \sJ(\divz^{m,\bullet})$, so $M_m = \omega(\sF)_m$, and let $\sF' := \sJ(\divz^{n+m,\bullet})$.  By Lemma~\ref{lem:Mmult}, $T(\divd)_n$ is the degree $n$ part of $\omega (\tau^{-(n-1)} G_\divd \circ \dots \circ G_\divd(\sO_X))$, where $G_\divd$ is the functor defined in Notation~\ref{not:Gd}.  
 Thus, by Lemma~\ref{lem-compose}, 
\[ M_m T(\divd)_n \subseteq \bigl( \omega(\tau^{-(m+n-1)} G_\divd \circ \dots \circ \tau^{-m}G_\divd(\sF)) \bigr)_{m+n}.\]
To show that $M$ is a right $T(\divd)$-module, it therefore suffices to show that 
\beq\label{foo}
\tau^{-(m+n-1) } G_\divd \circ \dots \circ \tau^{-m}G_\divd(\sF) \subseteq \sF'.
\eeq

By Lemma~\ref{lem-Fcomp2}, 
$\tau^{-(m+n-1)} G_\divd \circ \dots \circ \tau^{-m}G_\divd(\sF)$ is defined by 
$\divw^\bullet $, where
\[  \textstyle \divw^i = \min_{j=0}^i (\tau^{-j} (\divz^{m, i-j})+\tau^{-m-j}(\divd_{n-j})).\]
To show \eqref{foo}, we must show that $\divw^i \geq \divz^{n+m,i}$.  It therefore suffices to show for all $0 \leq j \leq i$ that
\beq\label{bar}  \divr =  \tau^{-j} (\divz^{m,i-j})+ \tau^{-m-j}(\divd_{n-j}) \ \geq\ \divz^{n+m,i}.\eeq 
Now, 
\[ \divr \geq \tau^{-i}(\divd_{m+j-i}-\divy+\tau^{-m-j+i}(\divd_{n-j})) = \tau^{-i}(\divd_{n+m-i}-\divy).\]
Since $\divr$ is effective,
$ \divr \geq \max(0, \tau^{-i}(\divd_{n+m-i}-\divy))  = \divz^{n+m,i}$.
  Thus \eqref{bar} holds.

Finally, since $\deg \divz^{n,i} \leq \deg \divd_{n-i} = (n-i)(\deg \divd) < (n-i)\mu$,   Lemma~\ref{hs-lem}(3) implies  that 
$\overline{M}_n  = H^0(E, \mc{M}_n(-\divz^{n,0}))_n$ for each $n \geq 0$.  
By assumption $\divy \leq \divd_k \leq \divd_n$ for all $n \geq k$, 
and so it follows that $\divz^{n,0} = (\divd_n -\divy)_{+} = (\divd_n - \divy)$ for all $n \geq k$, and hence 
$\overline{M}_{\geq k} = \bigoplus_{n \geq k} H^0(E, \mc{M}_n(-\divd_n + \divy))$.
\end{proof}

%%%%%%%%%%%%%%%%%%%%%%%%%%%%%%%%%%%%%

\section{Relating left and right ideals}\label{RIGHT-LEFT}

The main assumption of this paper, Hypothesis~\ref{hyp:main}, is left-right symmetric and so one can define left ideals of $T$ 
analogous to the $J(\divd^\bullet)$, with similar properties.
    What is less obvious  is  that these left ideals are closely related to the corresponding right ideals.
    In this section we study this relationship.  This   
     will  be useful  for the study of two-sided ideals of $T$,  
which will be the theme of  the last two sections of this paper.

Throughout the section, let $T$ be a ring satisfying Hypothesis~\ref{hyp:main}, with $B = T/gT \cong B(E, \mc{M}, \tau)$.   
Recall that the equivalence of categories $\coh E \to \rqgr B$ is given by  $\mc{F} \mapsto \bigoplus_{n \geq 0} H^0(E, \mc{F} \otimes \mc{M}_n)$
while the equivalence   $\coh E \to B \lqgr$  is given by 
$\mc{F} \mapsto \pi \bigl( \bigoplus_{n \geq 0} H^0(E, \mc{M}_n \otimes \mc{F}^{\tau^{n-1}}) \bigr)$.
For $q \in E$,  write 
\[ \sO_q' := \pi \bigl( \bigoplus_{n \geq 0} H^0(E, \sM_n \otimes \kk(q)^{\tau^{n-1}}) \bigr) \]
for
the object in $B \lqgr$ (or $T \lqgr$)  corresponding to    the skyscraper sheaf $ \kk(q)$ under the latter  equivalence.
There is a left-sided version of Theorem~\ref{thm:vdb}  giving an exact functor $\wt{( - )}_p: T \lqgr \to (C_p \lmod)$.
 
Typically, if $\tau^{-n}$ appears in some  right-sided formula, then   $\tau^n$ appears in  the left-sided analogue.
  Thus, for example, we define a \emph{left allowable divisor layering} to be a sequence of divisors 
  $\divd^\bullet = (\divd^0, \divd^1, \dots, \divd^{k-1})$ on $E$ that satisfies $\tau(\divd^{i-1}) \geq \divd^i$ for $1 \leq i \leq k-1$.  
 For such a layering $\divd^\bullet$,  we
define a saturated left ideal $J'(\divd^\bullet)=J'_T(\divd^\bullet)$ of $T$ in a way analogous to Definition~\ref{def:J}.
In more detail,  if $\divd^\bullet$ is  supported on $\mb{O}(p)$,  say with  $\divd^i = \sum_j a_{j, j-i} \tau^j(p)$
define $J'$ to be the saturated left ideal of $T$ for which 
$\wt{\pi(J')}_p$ is the left ideal of $C_p$ with $(x)^{a_{k, \ell}}$ in the $(k, \ell)$-spot.
  In general $J'(\divd^\bullet)$ is defined as before to be the intersection of the left ideals 
obtained by restricting the divisor layering to each relevant $\tau$-orbit. 

The left-sided analogue of Lemma~\ref{lem:layer}, with essentially the same proof,  is as follows.
\begin{lemma} 
\label{lem:left layer}
Let $\divd^{\bullet}$ be a left allowable divisor layering and let $J' = J'(\divd^{\bullet})$ and $M = T/J'$.
\begin{enumerate}
\item If $M^j = Mg^j/Mg^{j+1}$ is the $j^{\text{th}}$ layer of $M$, then in $B \lqgr$ we have 
\[
\pi(M^j) \cong  \pi \bigg(\bigoplus_{n \geq 0} H^0(E, \mc{M}_n \otimes (\mc{O}_E/\mc{O}_E(-\tau^{-n+1}(\divd^j)))) \bigg).
\]
\item $(\overline{J'})^{\sat} = \bigoplus_{n \geq 0} H^0(E, \mc{M}_n(-\tau^{-n+1}(\divd^0))).$
\item If $\divd^{\bullet} = (\divd)$ has length $1$, then $J'(\divd) = 
\bigoplus_{n \geq 0} \{ x \in T_n \st \bbar{x} \in H^0(E, \sM_n(-\tau^{-n+1}(\divd))) \}$.  \qed
\end{enumerate}
\end{lemma}
\noindent
 
We also have left-sided versions of Definition~\ref{def:M(k,D)} and  Lemma~\ref{lem:Mmult}.

\begin{definition}\label{def:M'(k,D)}
For any divisor $\divd$,  define $M'(k, \divd) = M'_T(k, \divd) =   J'(\divd^\bullet)$ for the divisor layering  
\[
\divd^0 = \divd + \tau(\divd) + \dots + \tau^{k-1}(\divd), \ \divd^1 = \tau(\divd) + \dots + \tau^{k-1}(\divd), \ \dots, \ \divd^{k-1} = \tau^{k-1}(\divd).
\]
\end{definition}
 
 \begin{lemma} 
\label{lem:Mmult2}
Let $n,m,k, l \in \mb{N}$.   Then 
\[
M'(\ell, \tau^{k-m}(\divd))_n M'(k, \divd)_m \subseteq M'(k+\ell, \divd)_{n+m}'.
\]
If $\mu - \deg \divd \geq 2$, this is an equality for all $m \geq k$ and $n \geq \ell$. If $\mu - \deg \divd = 1$, it is an equality 
for all $m \geq \max(2, k)$ and $n \geq \max(2, \ell)$.  \qed
\end{lemma}
\noindent 

For the rest of the section we consider another family of right ideals defined by divisor data, together with their left-sided versions, 
which will  be useful in   \cite{RSS2}.  We write 
$p_i = \tau^{-i}(p)$ for a closed point $p\in E$.

\begin{definition}\label{def-Q}
 Let $i \in \NN$ and let $0 \leq r \leq d \leq \mu \in \NN$.
For any fixed point $p \in E$,  define $Q(i,r,d,p)$ to be the right ideal $J(\divd^\bullet)$ of $T$
associated to the  divisor layering: 
\[ \divd^0 = dp + d p_1 + \dots + dp_{i-1},\  
\divd^1 = d p_1 + \dots + d p_{i-1}, \  \dots \ , \  \divd^{i-1} = r p_{i-1}. 
\]
Analogously,   define $Q'(i,r,d,p)$ to be the left ideal $J'(\divd^\bullet)$ of $T$
associated to the   left divisor layering: 
\[ \divd^0 = dp + d p_{-1} + \dots + d p_{-i+1},\   \divd^1 = d p_{-1} + \dots + d p_{-i+1}, \ \dots\ , \ \divd^{i-1} = r p_{-i+1}. 
\]
  \noindent Note that the divisor data for $Q(i,d,r,p)$ equals that  for $M(i, dp)$, 
except that $\divd^{i-1}$ may have smaller multiplicity.

 \end{definition}

   For $p \in E$, recall that  the right point module corresponding to the point $p$ is written
$P(p) = T/J(p)$, with $\pi(P(p)) = \mc{O}_p$ in $\rqgr T$.
Similarly, write $P'(p) = T/J'(p)$ for the left point module corresponding to $p$. 

\begin{lemma}  
\label{lem:Qfactor}
Let $i, r, d, n \in \mb{N}$, with $i < n$ and $1 \leq r \leq d \leq \mu$, and $p \in E$.  
\begin{enumerate}
\item $Q(i, r, d, p) \subseteq Q(i,r-1, d, p)$, with  factor
$
[Q(i,r-1,d,p)\big/Q(i,r, d,p)]_{\geq n} \cong P(p_{i-n-1})[-n].
$
\item Similarly, $Q'(i,r,d,p) \subseteq Q'(i, r-1, d, p)$ with % factor
\[
[Q'(i,r-1,d,p)\big/Q'(i,r, d,p)]_{\geq n} \cong P'(p_{-i+n+1})[-n].   
\]
\end{enumerate}
\end{lemma}

\begin{proof}
We prove only Part~(1).  By  Lemma~\ref{lem:layer}, the image of the module $W=\displaystyle \frac{Q(i, r-1, e, p)}{Q(i, r, e, p)}$  in $X = \rqgr T$  is a 
copy of   $\mc{O}_{p_{i-1}}$.
This simple object is also equal to $\pi(P(p_{i-n-1})[-n])$, so the required isomorphism holds
in large degree.  Now use Lemma~\ref{hs-lem} to see that 
$W_{\geq n}$ has the same Hilbert series as $P(p_{i-n-1})[-n]$. Since each $Q$ is saturated,  $W$ is torsion-free  and so 
we do have an isomorphism in the claimed degrees.
\end{proof}

\begin{lemma}\label{lem:interQ}
Let $n > \ell \in \NN$ and let  $\divd=\sum e_p p$ be a nonzero effective divisor on $E$ of degree at most $\mu -1$, 
supported at points with distinct $\tau$-orbits.
Then
\[ \bigl( T_{\leq \ell} * T(\divd) \bigr)_n = 
 \bigcap\Big\{  Q(i,r, e_p ,p_j)_n \, \big|\,    p \in \divd,\ i \geq 1,\  \ell \leq j \leq n-i,\ 1 \leq r \leq e_p\Big\}.
\]
In addition,   $T_{\leq \ell}* T(\divd) = \bigcap_p T_{\leq \ell}*T(e_p p)$.
\end{lemma}
\begin{proof}
To prove the first  assertion it is enough to  consider the case when $\divd = e_p p$ is a single multiple point.
This  is an easy consequence of Lemma~\ref{lem:lattice}.
The final statement is immediate from the definitions.   
\end{proof}

For any $i \geq 2$, $e \geq 1$, and $q \in E$, the following identity   also follows immediately from Lemma~\ref{lem:lattice}:
\begin{equation}
\label{eq:needlater}
Q(i, 0, e, q) = Q(i-1, e, e, q) \cap Q(i-1, e, e, \tau^{-1}(q)).
\end{equation}

We are now ready to prove the main result of this section,   relating the right and left ideals defined above. 

\begin{proposition}
\label{lem:MandMprime}
 Fix an effective divisor $\divd$ on $E$, $n,k, r,m \in \NN$ with $r \leq m$,  and $p \in E$.
\begin{enumerate}
\item $\overline{M(1, \divd)}_n = H^0(E, \mc{M}_n(-\divd))_n = \overline{M'(1, \tau^{n-1}(\divd))}_n$.
\item $M(k, \divd)_n = M'(k, \tau^{n-k}(\divd))_n$. 
\item $Q(k,r,m,p)_n = Q'(k,r,m,\tau^{n-k}(p))_n$.
\end{enumerate}
\end{proposition}

\begin{proof}
(1)  This follows immediately from Lemmas~\ref{lem:layer}(3) and \ref{lem:left layer}(3).

(2)  When $k = 0$, the result is trivial.  We first prove by induction that, for each $k\geq 1$, (2) holds if $n \gg 0$.
Assume that  $k = 1$.  Since both $M(1, \divd)$ and $M'(1, \tau^{n-1}(\divd))$ contain $gT$
 (use Lemmas~\ref{lem:layer}(3) and \ref{lem:left layer}(3)),  it is enough to prove that 
 $\overline{M(1, \divd)}_n = \overline{M'(1, \tau^{n-1}(\divd))}_n$. This is 
precisely Part~(1).

Now for $k \geq 2$, we observe that 
\begin{align*}
M(k, \divd)_{n + r}\  =&\ M(k-1, \divd)_r M(1, \tau^{r - k + 1}(\divd))_n  
\ =\ M'(k-1, \tau^{r -k +1}(\divd))_r M'(1, \tau^{n + r -k}(\divd))_n \\
 &=\ M'(k, \tau^{n + r- k}(\divd))_{n+r},
\end{align*}
for all $n, r\gg 0$.  Here, the first equality follows from Lemma~\ref{lem:Mmult}, the 
second from the induction hypothesis, and the last from Lemma~\ref{lem:Mmult2}.
This proves the induction step, and so (2) holds for each $k$, provided $n \gg 0$.

To get the result for all $n$,   note first  that for  fixed $n \geq 0$, we have for all $\ell \gg 0$ that 
\begin{equation}\label{MandMequ}
T_{\ell} M(k, \divd)_n \subseteq M(k, \tau^{-\ell}(\divd))_{n+\ell} = M'(k, \tau^{n-k}(\divd))_{n + \ell},
\end{equation}
where the first inclusion is a consequence of Lemma~\ref{lem:Mmult}, and the equality holds since 
we have proven Part~(2) for large $n + \ell$.  Each $M'$ is  saturated by definition, and since 
\eqref{MandMequ}  holds for all $\ell \gg 0$ we get
$M(k, \divd)_n \subseteq M'(k, \tau^{n-k}(\divd))_n$.  An analogous argument using Lemma~\ref{lem:Mmult2}  gives the reverse inclusion.

(3) It is easy to prove using Lemma~\ref{lem:lattice} that 
\[
Q(k,r,m,p)_n = M(k, rp)_n \cap M(k-1, mp)_n \cap M(k-1, m\tau^{-1}(p))_n.
\]
Similarly, a left-sided version of  that lemma shows   that 
\[
Q'(k, r, m, q)_n = M'(k, r q)_n \cap M'(k-1, m q)_n \cap M'(k-1, m \tau(q))_n
\]
  for any point $q$.   Now the result follows from Part (2) by taking $q = \tau^{n-k}(p)$.
\end{proof}

One consequence of  Proposition~\ref{lem:MandMprime}   is that 
  $T(\divd)$ can be defined either using either  left or right ideals:
\begin{equation}\label{T-leftright}
T(\divd) = \bigoplus_{n \geq 0} M(n, \divd)_n = \bigoplus_{n \geq 0} M'(n, \divd)_n.
\end{equation}

%%%%%%%%%%%%%%%%%%%%%%%%%%%%%%
\section{Divisors and exceptional line modules}
\label{DIVISORS}

As always, assume that $T$ is an algebra satisfying Hypothesis~\ref{hyp:main}. The main goal of this section is 
  to construct the exceptional line module for a one-point 
blowup $T(p) \subseteq T$, thereby proving Proposition~\ref{iprop:subtle} from the introduction.   
This will be used  in the next section to understand  the ideal structure of more general  blowups $T(\divd)$
(see Theorem~\ref{8-special}, in particular).  The key tool in the proof will be the notion of the 
 divisor associated to a $T$-module and so    the 
first part of this section will concentrate on this concept.

In fact it will be useful to work more generally than Hypothesis~\ref{hyp:main}.  Thus to begin the 
section we fix the following notation.  Let $S$ be a cg $\kk$-algebra which is a domain with a   central 
element $g \in S_{\gamma}$ for some $\gamma \geq 1$, such that $S/gS \cong B(E, \mc{L}, \sigma)$ for some  
elliptic curve $E$, invertible sheaf $\mc{L}$ with $\deg \mc{L} \geq 1$, and  infinite order automorphism $\sigma$. 
Such a ring $S$ will be called \emph{Sklyanin-like}, since the obvious examples are indeed the quadratic and
 cubic Sklyanin algebras of dimension three. Note that  the Veronese ring $T = S^{(\gamma)}$  satisfies  Hypothesis~\ref{hyp:main}, with   
$\mc{M} = \mc{L} \otimes \mc{L}^{\sigma} \otimes \dots \otimes \mc{L}^{\sigma^{\gamma -1}}$ and $\tau = \sigma^{\gamma}$.

We now turn to the divisor of a module. This concept is extensively studied  in \cite{Rog09}, to which the reader is referred for more details.

\begin{definition}\label{def-C}
Let $B  = B(E, \mc{L}, \sigma)$ be  as above with quotient map $\pi: \rgr B \to \rqgr B$.   
 If $M\in \rgr B$  with $\GK M = 1$, then $\pi(M)$  corresponds to a torsion sheaf in $\coh E\sim\rqgr B$, which has a finite filtration 
by  skyscraper sheaves $\kk(q)$ for points $q \in E$.  Write $ \divc(M)=\divc_B(M)$ for the 
divisor $\sum c_p p$, where $c_p$ is the number of times $\kk(p)$ occurs in this filtration.

 \end{definition}

\begin{definition}\label{def-admiss}
Let $\Div E$ denote  the free abelian group generated 
by the closed  points of the elliptic curve $E$ and set  $\Pic E = \Div E/\Lambda$, where $\Lambda$ is the subgroup of divisors
 linearly equivalent to $0$.
It will be convenient to identify points of $E$ on the same $\sigma$-orbit,  so  let $\overline{\Div} \,E=\Div E/\Gamma$ where  $\Gamma$ is 
 the subgroup   generated by   $\{p - \sigma(p) : p \in E\}$.
   Similarly, let $\overline{\Pic}(E) = \Div E/(\Gamma + \Lambda)$.
   Since $\sigma$ has infinite order,  $\sigma(x) =x + \alpha$ is translation by some $\alpha\in E$ under the group law on $E$.  
It follows easily that $\overline{\Pic}( E) = \Pic(E)/\mb{Z} \alpha$.  

Let  $S$ be Sklyanin-like. Then  $M\in \rgr S$ is called  \emph{admissible} if $\GK M/Mg \leq 1$.  
In this case the \emph{divisor associated to $M$} is 
$$\Div M = \Div_SM = \divc(M/Mg) - \divc(\Tor^1_S(M, S/Sg))\in \Div E.$$
 For an admissible module $M$, let $\overline{\Div}(M)$ be the image of 
$\Div M$ in $\overline{\Pic}(E)$.
Finally, let $$\Div(S) = \{ \Div M \in \Pic E \st M\ \text{is admissible} \}$$ and let $\overline{\Div}(S)$ denote the image
of $\Div(S)$ in 
$\overline{\Pic}(E)$.

There are analogous concepts for left $S$-modules. 
If $M \in S \lgr$ is an admissible left module we write its divisor as $\Div^{\ell}(M)$, with image $\overline{\Div}^{\ell}(M)\in \overline{\Pic}(E)$.
Similarly, set    \label{def-admiss2}
 $$\Div^{\ell}(S) = \{  \Div^{\ell}(M) | M \in S \lgr\ \text{is admissible} \} \quad \text{and} \quad
\overline{\Div}^{\ell}(S) = \{ \overline{\Div}^{\ell}(M) | M \in S \lgr\ \text{is admissible} \}.$$
\end{definition}

A number of properties of $\Div M$, mostly coming from \cite{Rog09},  are collected in  the  next lemma.  In \cite{Rog09}, the assumption is made that 
$S/gS \cong B(E, \mc{L}, \sigma)$ where $\deg \mc{L} \geq 2$, but we wish to allow $\deg \mc{L} = 1$.  
In most cases  this makes no real difference and we can simply quote the
result from \cite{Rog09}.  In a few cases we need to give a different proof 
which also covers the case $\deg \mc{L} = 1$.

\begin{lemma}
\label{lem:divprops}
Let $S$    be Sklyanin-like.
\begin{enumerate}
\item $M \in \rgr S$ is admissible precisely when $\GK M \leq 2$ and $\GK \tg(M) \leq 1$.  
\item $\Tor^1_S(M, S/Sg) \cong \{ m \in M \st mg = 0 \}[-\gamma]$ as graded $S$-modules.  Thus, 
if $M$ is admissible and $g$-torsionfree then $\Div(M)$ is effective. If, in addition, $\GKdim M=2$, then $\Div(M)\not=0$. 
\item $\Div M[n] = \sigma^n(\Div M)$ for any $n \in \mb{Z}$.
\item If   $0 \to M \to N \to P \to 0$ is an exact sequence of admissible $S$-modules, then $\Div N = \Div M + \Div P$.
\item $\overline{\Div}(M) = 0$ if $M \in \rgr S$ is admissible with $\GK(M) = 1$.
\item $\overline{\Div}(S) = \{ \overline{\Div}(M) \st M \in \rgr S\ \mbox{is  $g$-torsionfree and } \GK(M) = 2 \} \cup \{0\}$.
\end{enumerate}
\end{lemma}
\begin{proof}
Parts (1-4) follow from  \cite[(8.1) and  Lemma~8.3]{Rog09}.  Part (5) is \cite[Lemma~8.7]{Rog09}.  
For Part (6), suppose that $M \in \rgr S$ 
is admissible.  Then $\GKdim M \leq 2$.  If $N$ is the largest submodule of $M$ with $\GKdim N=1$, 
then $M/N$ is either 0 or GK 2-pure.  If $M/N=0$ we are done, so assume $M/N$ is GK 2-pure.  Now $\overline{\Div}(M/N) = \overline{\Div}(M)$ by Parts (4) and (5). 
  Now $M/N$ is also admissible and so 
$\tg(M/N) = 0$ by Part~(1); thus $M/N$ has GK-dimension 2 and is $g$-torsionfree.  Part (6) follows.
\end{proof}

We now return to the setting of the rings $T$ satisfying Hypothesis~\ref{hyp:main}, so that $g \in T_1$.

\begin{lemma}
\label{C-lem} 
 Consider 
$R = T(\divd +p) \subseteq \wt{R} = T(\divd)$, where $p\in E$ and 
$\divd$ is effective with $\deg \divd < \mu -1$.  
If $M \in \rgr \wt{R}$ is admissible, then $M$ is 
also   admissible   as a right $R$-module and $\Div_R M=\Div_{\wt{R}}M$.
\end{lemma}

 \begin{proof}   Write
$R/gR = B = B(E, \mc{N}, \tau)  \subseteq \wt{R}/g\wt{R} = \wt{B} = B(E, \mc{M}, \tau)$, 
where $\mc{N} = \mc{M}(-p)$; thus $\deg \mc{N} \geq 1$. The key point is to show that, if 
$N\in \rgr \wt{B}$, then $N\in\rgr B$ with  $\divc_B(N)=\divc_{\wt{B}}(N)$
(this is the analogue of \cite[Lemma~3.5]{Rog09}).
By  \cite[(6.2.19) and (6.8.25)]{MR}, $N$ has a composition series where the factors are either  finite dimensional or 
1-critical $\wt{R}$-modules. By a routine induction, we can therefore assume that $N$ is a cyclic 1-critical right $\wt{B}$-module. 
As such, Lemma~\ref{lem:pt-crit}(3) implies that 
(up to a shift) $N$ is a $\wt{B}$-point module: $N  = \wt{B}/I$ where $I =  \bigoplus_{n \geq 0} H^0(E, \mc{I}_q \otimes \mc{M}_n)$
and $\mc{I}_q$ is the ideal sheaf defining $q$.   

 Suppose that $\wt{B}_n B_m \subseteq I$  for some $n,m \gg 0$.  
Then $H^0(E, \mc{M}_n) H^0(E, \mc{N}_m^{\tau^n}) \subseteq 
H^0(E, \mc{I}_q \otimes \mc{M}_{n+m})$.  Choosing $n,m$ large enough so that all sheaves in this equation are generated by their sections, 
gives $\mc{M}_n \mc{N}_m^{\tau^n} \subseteq \mc{I}_q \otimes \mc{M}_{n+m}$ 
and hence $\mc{N}_m^{\tau^n} \subseteq \mc{I}_q \otimes \mc{M}_{m}^{\tau^n}$.   But 
$\mc{M}_m/\mc{N}_m$ is  supported at finitely many points. Thus, 
for  $m \gg 0$, the equation $\mc{N}_m^{\tau^n} \subseteq \mc{I}_q \otimes \mc{M}_{m}^{\tau^n}$
only holds for finitely many $n\geq 0$.
Thus for fixed $m \gg 0$, one has $\wt{B}_n B_m \not\subseteq I$ and hence
$N_n B_m = N_{m+n}$ for all $n \gg 0$.  Thus $N_B$ 
is finitely generated.   

Let $Q(p) = \bigoplus_{n \geq 0} H^0(E, \kk(p) \otimes \mc{N}_n)$ be the $B$-point module corresponding to a point $p\in E$.
A similar calculation    shows 
that $\{x \in B_m \st N_n x  = 0 \} = H^0(E, (\mc{I}_q)^{\tau^{-n}} \otimes \mc{N}_m) $ as long as $m, n \gg 0$.  
Thus   for some  $n\gg 0$ we have $N_n B \cong Q(\tau^n(q))[-n]$ as  $B$-modules.
 Hence 
$\pi(N) = \pi(Q(\tau^n(q))[-n]) = \pi(Q(q))$ in $\rqgr B$ and so $\divc_B(N) = q$   as well.

 This proves the analogue of \cite[Lemma~3.5]{Rog09} and now the proof of 
  \cite[Proposition~9.5(1)]{Rog09} may be used   to prove the present result.
\end{proof}

\begin{definition}\label{line-defn1}
A graded right $R$-module $M$ is called a \emph{line module} if $M$ is finitely generated and has Hilbert series 
$h_M(t) = 1/(1-t)^2$.  
\end{definition}
\noindent  As with point modules, line modules are often assumed to be cyclic, but this is not really appropriate 
when $T$ is not generated in degree one.  

\begin{lemma}
\label{lem:crit}
Let  
$R = T(\divd)$ for some effective divisor $\divd$ with $\deg \divd < \mu$.  
\begin{enumerate}
\item
Let $M$ be a $g$-torsionfree,  admissible right $R$-module,  with $\GKdim M=2$ such that  $\Div M = p$ is  a single point.
Then $M$ is 2-critical.  
\item If $M$ is a $g$-torsionfree line module, then Part (1) applies and $M$ is 2-critical. 
\end{enumerate}
\end{lemma}
 
\begin{remark}  If $\mu - \divd \geq 2$ then any cyclic line module for $T(\divd)$ is $g$-torsionfree, 
by \cite[Lemma~8.9]{Rog09}.
\end{remark}

\begin{proof}
(1) By Lemma~\ref{lem:divprops}(1), $M$ is 2-pure.  Thus, if $M$ is not critical, then it has a graded submodule $N$ with 
 $\GKdim N=2$ such that $M/N$ is 2-critical.
  Since $M/N$ is still admissible,   Lemma~\ref{lem:divprops}(1) implies that $M/N$ is  $g$-torsionfree.
Then both $\Div N$ and $\Div M/N$ are nonzero effective divisors and $p = \Div M = \Div N + \Div M/N$ by Lemma~\ref{lem:divprops}(2,4).
This  is obviously impossible.

(2)  Since $M$ is $g$-torsionfree,   $M/Mg$ has the Hilbert series  $1/(1-t)$, and so $P = M/Mg$ is a point module. 
Thus $M$ is admissible.  
By Lemma~\ref{lem:pt-crit}(1,2),  
$P_{\geq n} \cong P(q)_{\geq n}$ for some point module  
$P(q)$, and so $\pi(M) = \pi(P(q)) = \mc{O}_q$ and $\Div M = q$ is a point.   Thus the hypotheses of Part (1) hold.
\end{proof}
 
  We next   show that $\overline{\Div}(T)=\overline{\Div}^\ell (T)$. The proof will be homological, using the fact that
    $T$ is 
   Auslander-Gorenstein and CM  by Proposition~\ref{prop:foo}.
Given a (right or left) module $M$ over a ring $A$, set 
 $E^i(M) =E^i_A(M)=  \Ext^i_A(M,A)$  for $i\geq 0$.
We begin with some preparatory lemmas.

\begin{lemma}\label{lem:dualB}
 Let $B= B(E, \sM, \tau)$ where $E$ is elliptic and $\deg \sM \geq 2$. 
For $q \in E$, define
\[ I(q) = \bigoplus H^0(E, \sM_n(-q)), \quad \quad I'(q) = \bigoplus H^0(E, \sM_n(-\tau^{-n+1}(q))).\]
\begin{enumerate}
\item Let $P(q) = B/I(q)$ and $P'(q) = B/I'(q)$.  
Then $\pi(E^1_B(P(q))) \cong \pi(P'(\tau^{-2}(q))[-1])$ in $B \lqgr $.  
\item If $P\in \rgr B$ with $\GKdim(P)\leq 1$ then 
$ [ \divc(P)]=[\divc(E^1_B(P))]$ in $\overline{\Pic}(E).$
\end{enumerate} 
\end{lemma}

\begin{proof}  (1)
 Let $K = E^0(I(q)) = \Hom_B(I(q), B)$; thus $E^1(P(q))\cong K/B$. 

We claim that  $K \ppe K'$ where $K'= \bigoplus_{r\geq 0} H^0(E, \sM_r(\tau^{-r}(q)))$.
Indeed, by \cite[Theorem~1.3]{AV}, $K =\bigoplus_{r\geq 0} H^0(E, \mathcal P_r)$
 for some invertible sheaves $\mathcal P_r\subset \kk(E)$
  that are ample and globally generated for $r\gg 0$.
Now expand the equation $K_nI(q)_r \subseteq B_{n+r}$ in terms of sheaves, using the multiplication rules for $B$.
 This shows that $\sP_r$ is defined by 
$$ 
  \sP_r \otimes \Bigl( \sO(\mathbf{q}) \otimes \mathcal M_n \Bigr)^{\tau^r}  
\   \subseteq  \            \mathcal{M}_{n+r}.
$$
Therefore,
$ 
\mathcal P_r \ =   \mathcal M_r\otimes \sO(-\tau^{-r}(\mathbf{q})) .$
    Taking sections gives the required assertion.  
 
By Riemann-Roch,   $\dim_{\kk} K'_r= \dim \sM_r(\tau^{-r}(q))=1+\dim B_r$, and so 
$\dim_{\kk}(K'/B)_r=1$  for all $r\geq 1$. Also, by Lemma~\ref{lem:sec-mult},  
$(K'/B)_{\geq 1}$ is generated in degree one and so it is a shifted point module.
Since  
\[H^0\left(E, \sM_n\left(-\tau^{-n-1}(q)\right)\right) 
H^0(E, \sM_1\left(\tau^{-1}(q)\right)^{\tau^n}) \subseteq H^0(E, \sM_{n+1})\]
it follows  that 
$I'(\tau^{-2}(q)) \subseteq  \Ann_B((K'/B)_1)$. As $B$ is generated in degree one,
$\Ann_B((K'/B)_1)_r\not=B_r$, for any $r\geq 1$.
Since $\dim B/I'(\tau^{-2}(q))_m=1$ for $m\geq 1$ it follows that  $I'(\tau^{-2}(q)) =  \Ann_B((K'/B)_1)$
  and $K'_{\geq 1} \cong B/I'(\tau^{-2}(q))[-1]$, as required.

(2)  For $Q\subseteq P$,  there is an exact sequence $0\to E^1(P/Q)\to E^1(P)\to E^1(Q)\to E^2(P/Q).$
By \cite[Theorem~6.6]{Le},  $B$ is Auslander-Gorenstein and CM, whence  $\GKdim(E^2(P/Q)) =2-j(E^2(P/Q))\leq 0$. 
Thus, since $\divc$ is computed in $\rqgr B$ and is additive on short exact sequences, 
 $\divc(E^1(P))=\divc(E^1(Q))+\divc(E^1(P/Q))$ and, analogously, for $\divc(P)$. Thus, 
by induction, we can reduce to the case when $P$ is 1-critical and hence,
by Lemma~\ref{lem:pt-crit},  
equals  a shifted point module (up to finite dimensions). The result now  follows from Part~1.
\end{proof}

 \begin{lemma}\label{thou6} Let $R=T(\divd)$ where $ \deg \divd<\mu$.
Let  $M$ be a $g$-torsionfree right $R$-module.  Then, for $i \geq 0$, there is an isomorphism of left $\overline{R}$-modules
\[
\Ext^i_R(M, \overline{R}) \cong \Ext^i_{\overline{R}}(M/Mg, \overline{R}).
\]
\end{lemma}
\begin{proof}   By \cite[Case 3, p.118]{CE}, there is a homomorphism 
$\theta:   \Ext^i_{\overline{R}}(M  \otimes_R \overline{R}, \overline{R})   \to  \Ext^i_R(M, \overline{R}).$
Since $\theta$  is obtained by taking a projective  resolution   of the $R$-module $M$, 
it is a homomorphism of right $\Rbar$-modules.
As $M$ is $g$-torsionfree,  $ \Tor_1^R(M, \overline{R}) = 0$ by Lemma~\ref{lem:divprops}(2), 
while $ \Tor_p^R(M, \overline{R})   = 0$ for $p>1$ since $\pd_R \overline{R} = 1$.  
 Thus, as noted in \cite[ Proposition~VI.4.1.3, p.~118]{CE},  $\theta$ is an isomorphism.  
\end{proof}

\begin{proposition}\label{prop:equaldiv}
Let  
$M$ be an admissible right $T$-module. 
Then  $E^1(M)$ is an admissible left $T$-module, and 
$\bbar{\Div}(M) = \bbar{\Div}^\ell(E^1(M))$.   
\end{proposition}

\begin{proof}  
Let $M$ be an admissible right $T$-module, and let $N\subseteq M$ be the largest submodule with $\GKdim N \leq 1$.
Then $j(N)\geq 3-1=2$ by the CM condition and hence $E^1(N)=0$. Thus, by Lemma~\ref{lem:divprops}(5),
$  \overline{\Div}(E^1(N))=0= \overline{\Div}(N)$. 
From  the exact sequence
$0\to E^1(M/N)\to E^1(M)\to E^1(N)$  we obtain 
$E^1(M)\cong E^1(M/N)$. Thus, by Lemma~\ref{lem:divprops}(4,5), $
 \overline{\Div}(M)= \overline{\Div}(M/N)$ and  $ \overline{\Div}(E^1(M))= \overline{\Div}(E^1(M/N))$. So, we can replace $M$ by 
 $M/N$ and assume that $M$ is 2-pure and hence $g$-torsionfree
 by Lemma~\ref{lem:divprops}(1).
 In this case \cite[Corollary~1.22]{Bj2} implies that $j(E^2(M))\geq 3$, whence $\GKdim(E^2(M))\leq 0$.

  Now  consider the exact sequence
\begin{equation}\label{div-equ} \xymatrix{
 \Hom_T(M,B) \ar[r] & E^1(M)[-1] \ar[r]^{g\cdot } & E^1(M) \ar[r] & \Ext^1_T(M, B) \ar[r] & E^2(M)[-1].}
\end{equation}
Since $M$ is now assumed $g$-torsionfree, it follows from \eqref{div-equ} and Lemma~\ref{thou6} that
\[ E^1(M)/gE^1(M)\hookrightarrow \Ext^1_T(M, B)\cong \Ext^1_B(M/Mg,B).\]
Thus final module is certainly a Goldie torsion $B$-module and so $\GKdim(\Ext^1_B(M/Mg,B))\leq 1$.
In particular, $E^1(M)$ is also admissible. 

By admissibility of $M$, $\Hom_T(M, B) = 0 $
 and so \eqref{div-equ} implies that  $E^1(M)$ is $g$-torsionfree.
  Thus, by Lemma~\ref{lem:divprops}(2),  $\Tor^1_T(E^1(M),T/gT)=0=\Tor^1_T(M,T/gT)$.
   Hence $\Div M=\divc(M/Mg) $ and
  $\Div E^1(M)=\divc(E^1(M)/gE^1(M)) $.
 Since $\dim_{\kk}(E^2(M)[-1])<\infty$, it follows from 
   \eqref{div-equ}  that
$\pi\left( E^1(M)/gE^1(M)\right) \cong \pi\left(  \Ext^1_T(M,B)\right) $.
Therefore, by Lemma~\ref{thou6},
  \[ \Div E^1(M) = 
 \divc(E^1(M)/gE^1(M)) = \divc( \Ext^1_T(M, B) )=\divc ( \Ext^1_B(M/Mg, B)) =\divc (E^1_B(M/Mg)) ,\]
while $\Div M  = \divc ( M/Mg).$
Thus, to prove the result, we need only prove that $[\divc_B(P)]=[\divc_B(E^1_B(P))]$
in $ \overline{\Pic}(E)$
for any   $P\in\rgr B$ with $\GKdim(P)\leq 1$. This is  Lemma~\ref{lem:dualB}(2).
 \end{proof}

\begin{corollary}\label{cor:leftright} 
 $\bbar{\Div}(T) = \bbar{\Div}^{\ell}(T)$. 
\end{corollary}
\begin{proof}
 By  Proposition~\ref{prop:equaldiv}  $\bbar{\Div}(T) \subseteq \bbar{\Div}^{\ell}(T) $
 and so we are done by symmetry.
\end{proof}

We now come to  the second main result of this section, which studies the structure of $\wt{R}/R$ as an $R$-module for a one point 
blowup $R =   \wt{R}(p)$.
  The result is a generalisation of \cite[Lemma~9.1]{Rog09}, with a different proof to allow 
for the fact that $R$ may not be generated in degree $1$.
 
\begin{proposition} \label{prop:subtle}
Let   $R = T(\divd + p) \subseteq \wt{R}= T(\divd)$ for some effective divisor $\divd$
with $\deg \divd < \mu - 1$.
\begin{enumerate}
\item 
There is a $g$-torsionfree line module $L$, with $\Div(L) = \tau(p)$,  so that 
$\wt{R}/R \cong \bigoplus_{i \geq 1} L[-i]$ as right $R$-modules.
\item  The module $L$ is cyclic except in the special case where $\deg \divd = \mu -2$ and $\overline{R} = B(E, \mc{L}, \tau)$ with 
$\mc{L} \cong \mc{O}_E(\tau(p))$, in which case $L$ is generated in degrees $0$ and $1$.
\item If  $L$ is cyclic, then $L \cong R/J$ where $J = \bigoplus_{n \geq 0} M_{\wt{R}}(n+1, \tau(p))_n$.
\end{enumerate}
\end{proposition}

\begin{remark}\label{line-defn} The line module $L$    is called 
 the \emph{right exceptional line module} for the blowup $R \subseteq \wt{R}$. 
 By symmetry there is also a $g$-torsionfree left line module $L'$, so that $\wt{R}/R \cong \bigoplus_{i \geq 1} L'[-i]$ as left $R$-modules.  
We call $L'$ the \emph{left exceptional line module} for the blowup $R \subseteq \wt{R}$.
\end{remark}
 
 \begin{proof}
(1, 2) In this proof all modules  $J(\divd^\bullet)$ and $M(k, \divd)$ are defined with respect to $\wt{R}$, while 
$\sO_X = \pi(\wt{R}) \in X = \rqgr \wt{R}$.  
In particular, 
for all $n \geq 0$ we have the right $\wt{R}$-ideal $M(n,p) = J(\divd^\bullet)$ defined   by 
the divisor layering $\divd^i = \sum_{k=i}^{n-1} \tau^{-k}(p)$ for each $0 \leq i \leq n-1$. By definition 
  $R = \bigoplus_{n \geq 0} M(n,p)_n$.

For  $n \geq 0$  define a divisor layering $\divc(j,n)^\bullet$ by
$\divc(j,n)^i = \sum_{i \leq k \leq n-1, k \neq i + j} \tau^{-k}(p)$ for $0 \leq i \leq n-1$.  
(This is similar to the divisor layering for $M(n,p)$, except that the vanishings corresponding 
to the row $j$ (if any) in  the ring $C_p$ are omitted.)
Let $N(j,n) = J(\divc(j,n)^\bullet)$ and define a right $R$-module by
\[
P^{(j)} = \bigoplus_{n \geq 0} N(j,n)_n \supseteq R = \bigoplus_{n \geq 0} M(n,p)_n.
\]
Since $\tau^{-1}(\divc(j, n-1)^i) = \divc(j,n)^{i+1}$ Lemma~\ref{hs-lem}(2) implies that each $P(j)$ is $g$-divisible.

We claim that $P^{(j)}$ is a right $R$-submodule of $\wt{R}$.  Using Lemma~\ref{lem:Mmult1} and Notation~\ref{not:Gd}, write 
$M(n,p) = \omega(\tau^{-n+1}(G_p) \circ \dots \circ G_p(\mc{O}_X))$.
 In order to show that $P^{(j)}_m R_n \subseteq P^{(j)}_{m+n}$, it suffices   by Lemma~\ref{lem-compose}(2)   to show that 
\[
\tau^{-m-n+1}(G_p) \circ \dots \circ \tau^{-m}(G_p)[\mc{J}(\divc(j,m)^\bullet)] \  \subseteq \  \mc{J}(\divc(j,m+n)^\bullet).
\]
Using Lemma~\ref{lem-Fcomp2} this is   an easy exercise.

Let $M^{(j)} = P^{(j)}/R$ for each $j \geq 0$.  The Hilbert series  
$h_{P^{(j)}}(t)$ can be calculated by Lemma~\ref{hs-lem}, from which one sees that $M^{(j)}$ 
has the Hilbert series
$t^{j+1}/(1-t)^2$   of a shifted line module.
By  Theorem~\ref{thm:Rnoeth}, $R$ is $g$-divisible,  whence $M^{(j)}$ is $g$-torsionfree.

We next verify that $\wt{R}/R \cong \bigoplus_{j\geq 0} M^{(j)}$ as right $R$-modules.   First, the Hilbert series 
of both sides are easily seen to be the same, so it is enough to show that the submodules 
$M^{(i)}$ are independent.  This can be done degree by degree.  In degree $n$, it amounts to showing that for $0 \leq j \leq n-1$, 
\begin{equation}\label{subtle3}
P^{(j)}_n \bigcap \sum_{\{ 0 \leq k \leq n-1, k \neq j \}} P^{(k)}_n \subseteq M(n,p)_n = R_n.
\end{equation}
By Lemma~\ref{lem:lattice},  the right $\wt{R}$-ideal  $\sum_{\{ 0 \leq k \leq n-1, k \neq j \}} N(k,n)$
 has image $\mc{J}(\divb^\bullet) $ in $ \rqgr \wt{R}$, where 
\[\divb^\bullet = \min_{\{0 \leq k \leq n-1, k \neq j\}} \divc(k,n)^\bullet = (\tau^{-j}(p), \tau^{-j-1}(p), \dots, \tau^{-j-n+1}(p)).\]
  Since $J(\divb^\bullet)$ is saturated,  in order to prove \eqref{subtle3}, it is enough to show that 
$N(j,n)_n \cap J(\divb^\bullet)_n \subseteq R_n$. But $N(j,n) \cap J(\divb^\bullet) = M(n,p)$ by another 
use of Lemma~\ref{lem:lattice}.   Thus $\wt{R}/R \cong \bigoplus_{j \geq 0} M^{(j)}$ as claimed.

We next fix $j$ and show that $M = M^{(j)}\in \rgr R$.   We first consider $M/Mg = M/Mg$.
By the $g$-divisibility of $R$, we have $R \cap g P^{(j)} = gR$.  It follows that $(R+gP^{(j)})/gP^{(j)} \cong R/(R\cap gP^{(j)}) = \bbar{R}$ and that
\[
\bbar{M} \ = \  \frac{P^{(j)}}{R +gP^{(j)} }\ \cong \ \frac{P^{(j)}/g P^{(j)}}{(R+gP^{(j)})/gP^{(j)}} \ \cong \ {\bbar{P^{(j)}}}/{\bbar R},
\]
where we use  the  $g$-divisibility of $P^{(j)}$  for the last isomorphism.
Thus we need to look at the  generators of 
\[
\sum_{n \geq j+1} \overline{P^{(j)}}_n/\overline{R}_n= \bigoplus_{n \geq j+1} H^0(E, \mc{L}_n(\tau^{-j}(p))) /H^0(E, \mc{L}_n) 
\]
as a right $\overline{R} = \bigoplus_{n \geq 0} H^0(E, \mc{L}_n)$-module.

Consider when 
$\overline{M}_n \overline{R}_m  \subseteq \overline{R}_{n+m}$ might hold for $n \geq j+1, m \geq 0$.  
Using the multiplication  in $B$ in terms of sections of sheaves, 
this can be written as
\begin{equation}\label{subtle2} 
H^0(E, \mc{L}_{n}(\tau^{-j}(p))) H^0(E, \mc{L}_m^{\tau^n})
 \subseteq H^0(E, \mc{L}_{n+m}).
\end{equation}
As long as either $\deg \mc{L} \geq 2$, or $\deg \mc{L} = 1$ and $m \geq 2$,
  Lemma~\ref{lem:sec-mult} implies that     $\mc{L}_{n+m}(\tau^{-j}(p))$ is generated by the sections 
$H^0(E, \mc{L}_{n}(\tau^{-j}(p))) H^0(E, \mc{L}_m^{\tau^n})$,
 and so it cannot be contained in $H^0(E, \mc{L}_{n+m})$. Thus \eqref{subtle2}  cannot hold.
 On the other hand, if $\deg \mc{L} = 1$ and $m = 1$, then $\dim_{\kk} H^0(E, \mc{L}_m^{\tau^n}) = 1$ and so 
both $H^0(E, \mc{L}_{n}(\tau^{-j}(p))) H^0(E, \mc{L}_m^{\tau^n})$ and $H^0(E, \mc{L}_{n+m})$ are 
$n+1$-dimensional vector spaces; they will be equal if and only if $\mc{L}_n(\tau^{-j}(p)) \cong \mc{L}_{n+1}$.
This is equivalent to $\mc{L} \cong \mc{O}_E(\tau^{n-j}(p))$.
In conclusion, $\overline{M}$ is cyclic and generated in degree $j +1$ unless $\mc{L} \cong \mc{O}_E(\tau(p))$, in 
which case it is generated in degrees $j + 1$ and $j + 2$.  Thus, by the graded Nakayama's lemma,   the same is true for $M$.
  
Since $M^{(j)}$ is finitely generated, it is a shifted $g$-torsionfree line module.  
Thus,  by Lemma~\ref{lem:crit}(2), $M^{(j)}$  is critical of GK-dimension 2. 
 It is easy to see from the calculation
 above that   $\Div M^{(j)} = \tau^{-j}(p)$.

Now let $Q^{(0)} = R$ and for each $j \geq 1$ set $Q^{(j)} = P^{(0)} + P^{(1)} + \dots + P^{(j-1)}$.  Since $\wt{R}/R=\bigoplus M^{(j)}$, 
we see that $Q^{(j+1)}/Q^{(j)} \cong M^{(j)}$ as right $R$-modules.  
We claim that $Q^{(j)} = \sum_{n \geq 0} M(n-j, \tau^{-j}(p))_n$ (recall that $M(i,\divd) = \wt{R}$ if $i \leq 0$.)
To prove this, note that $P^{(k)}_n = N(k,n)_n \subseteq M(n-j, \tau^{-j}(p))_n$ for all $0 \leq k \leq j-1$ 
and each $n \geq 0$, simply because 
the left hand side is defined by a larger divisor layering.  This implies that
$Q^{(j)} \subseteq \sum_{n \geq 0} M(n-j, \tau^{-j}(p))_n$. But both of these modules have the same Hilbert series:  
the Hilbert series of $Q^{(j)}$ is clear because $Q^{(j)}/R \cong M^{(0)} \oplus \dots \oplus M^{(j-1)}$, and the Hilbert series 
of $\sum_{n \geq 0} M(n-j, \tau^{-j}(p))_n$ follows from Proposition~\ref{prop-M}.
This proves the claim.

Let $L  = M^{(0)}[1]$.  We still must show that $M^{(i)} \cong L[-i-1]$ for $i \geq 2$.  
To do this, we first show that $\wt{R}_1 Q^{(j)} \subseteq Q^{(j+1)}$.  
 In   degree $n+1$, we must therefore show that 
\beq\label{toget}
 M(0, \tau^{-j-1}(p))_1 M(n-j, \tau^{-j}(p))_n  = \wt{R}_1 Q^{(j)}_n \subseteq  Q^{(j+1)}_{n+1} = M(n-j, \tau^{-j-1}(p))_{n+1}.
\eeq 
Clearly,  \eqref{toget} is   a special case of Lemma~\ref{lem:Mmult}(1). 

It is   clear that $\wt{R}_j \subseteq 
Q^{(j)}_j = M(0, \tau^{-j}(p))_j$, while $\wt{R}_{j+1} \nsubseteq Q^{(j)}_{j+1} = M(1, \tau^{-j}(p))_{j+1}$.
If it happened that $\wt{R}_1 Q^{(j)} \subseteq Q^{(j)}$, then since $\wt{R}$ is generated in degree $1$ (see Theorem~\ref{thm:Rnoeth}),
we would have $\wt{R}_1 \wt{R}_j  = \wt{R}_{j+1} \subseteq Q^{(j)}$, a contradiction.  Thus 
$\wt{R}_1 Q^{(j)} \nsubseteq Q^{(j)}$.  We can now choose $z \in \wt{R}_1$ such that $z Q^{(j)} \nsubseteq Q^{(j)}$, but 
$z Q^{(j)} \subseteq Q^{(j+1)}$.   Then left multiplication 
by $z$ induces a right $R$-module map 
\[
\theta:  M^{(j-1)}[-1] \cong \bigl( Q^{(j)}/Q^{(j-1)} \bigr) [-1] \to Q^{(j+1)}/Q^{(j)} \cong M^{(j)},
\]
where $\theta\not=0$ by the choice of $z$.  But each $M^{(j)}$ is $2$-critical.  Thus  $\ker \theta = 0$ 
since otherwise   Im$(\theta)\subset M^{(j)}$ is a nonzero submodule of GK dimension 1.  
Since both sides have the same Hilbert series, $\theta$ is  therefore an  isomorphism and  $M^{(j)}[-1] \cong M^{(j+1)}$
for each $j$. In particular,  each $M^{(j)} \cong L[-j-1]$ with  
 $\Div L = \tau(\Div M^{(0)}) = \tau(p)$.

(3)  Let $J = \bigoplus_{n \geq 0} M(n+1, \tau(p))_n$.  By Lemma~\ref{lem:Mmult}(2),  $J$ is a right ideal of $R$
while Lemma~\ref{hs-lem}(3) implies  that  $\dim_\kk (R/J)_n \leq n+1$ for each $n$.  Now if $L[-1] = P^{(0)}/R$ as  
 above, then Lemma~\ref{lem:Mmult}(1) implies that
\[
P^{(0)}_1 J_n = M(0, p)_1 M(n+1, \tau(p))_n \subseteq M(n+1, p)_{n+1} = R_{n+1}.
\]
  Since   $L$ is cyclic, $R/J$ surjects onto $L$.  Thus they both have Hilbert series $1/(1-t)^2$, and $L \cong R/J$.
\end{proof}

%%%%%%%%%%%%%%%%%%
%%%%%%%%%%%%%%%%%
\section{Minimal \spe\ ideals}\label{8POINTS2}

Throughout this section $T$ will be a ring satisfying Hypothesis~~\ref{hyp:main}.
The aim of this section is two-fold. In both cases we will work with a slightly more  restrictive class of algebras $T$,
 but it will include the algebras arising as Veronese subalgebras of Sklyanin algebras.
The first result will be to show that $\Div T$ is countable.  
  This will   then be used to prove that
 $T$ has a relatively rigid ideal structure. This in turn will   crucial  to the proof  of Theorem~\ref{ithm:maxnoeth}  in \cite{RSS2}.

  To make this more formal, assume that $S$ is a Sklyanin-like algebra, as defined in Section~\ref{DIVISORS}.
A graded ideal $I$ of $S$  is called  \emph{\spe} if $\GK S/I = 1$. \label{sporadic-defn}
 A \spe\ ideal $I$ of   $S$   is a \emph{minimal \spe\ ideal} of $S$
 if every other \spe\ ideal $H$ satisfies $H \supseteq I_{\geq m}$ for some $m$.  When $S$ has no ideal $I$ with $\GKdim S/I=1$, we declare that $S$ itself is the minimal \spe\ ideal.
 As we will show in this section, this property does hold for 
 our main examples.

    Many of the known examples of   a ring $T$ satisfying Hypothesis~\ref{hyp:main}    
     arise as Veronese rings of Artin-Schelter regular algebras $S$.
 As we next show,  $\Div(T)$ is  countable in these cases.
This result is closely related to results in  \cite{Aj} where the divisor of a module was originally defined.  

\begin{lemma}
\label{lem:countable}
Suppose that $S$ is a Sklyanin-like algebra of finite global dimension and let $T = S^{(\gamma)}$ be the Veronese ring that satisfies 
Hypothesis~\ref{hyp:main}.  Then $\Div(T)$ is countable. 
\end{lemma}

\begin{proof}
Given an admissible $S$-module $N$,  the   first  two paragraphs of the proof of  \cite[Lemma~8.8]{Rog09}  show that 
the Veronese $N^{(\gamma)}$ is an admissible $T$-module, with $\Div_S N = \Div_T N^{(\gamma)}$. Conversely, given 
an admissible $T$-module $M$, then  $N=M\otimes_TS$ is an admissible $S$-module  with $N^{(\gamma)}\cong M$. 
Thus, it suffices to prove that $\Div S$ is countable. 
Recall from Definition~\ref{def-admiss} that $\overline{\Pic}( E) = \Pic(E)/\mb{Z} \alpha$, 
where $\tau$ is translation by $\alpha \in E$. Thus $\Pic E \to \overline{\Pic}\, E$ is a countable-to-one function  
 and so it is enough to prove that $\overline{\Div} \,S$ is countable.

 By Lemma~\ref{lem:divprops}(4,5), $\overline{\Div}\,S$ is determined by the divisors of   admissible 
 $g$-torsionfree modules  $M\in \rgr S$. For such an $M$,  pick  a   finite graded  free  
resolution $F_{\bullet} \to M$.  By   Lemma~\ref{lem:divprops}(2)  and   induction,  $F_{\bullet} \otimes_S S/gS$
is exact and hence is
 a graded free resolution of  $M/gM$ as a module over $\overline{S} = B = B(E, \mc{L}, \sigma)$.
Passing to $\rqgr B \simeq \coh E$ we get a locally free resolution $\mc{F}_{\bullet} \to \mc{M}$, where 
$\mc{M}=\pi(M/Mg)$ is a torsion sheaf whose   support, counted with multiplicity, gives $\Div M$.

The image $[\mc{M}]$ of $\mc{M}$ in the Grothendieck group $K(E)$ only depends upon the graded
Betti numbers of  $F_\bullet$. Since there are only countably many choices of these Betti numbers, 
 there are only countably many choices for the  $ [\mc{M}]$. However, 
$ [\mc{M}]= (\mc{O}_E(\Div(M)),0)$, under  the decomposition  $K(E)\cong \Pic E \bigoplus \mb{Z}$ from \cite[Exercise~II.6.11]{Ha}
and  so there are only countably many choices for $ \Div(M)$.
 \end{proof}

We now turn to an arbitrary ring $T$ satisfying Hypothesis~~\ref{hyp:main}.

\begin{proposition}
\label{prop-induct}
Let $T$ satisfy Hypothesis~\ref{hyp:main}, and let $R = T(\divd + p) \subseteq \wt{R} = T(\divd)$ where $\deg \divd < \mu - 1$.
Then $\overline{\Div}(R) = \overline{\Div}(\wt{R}) + \mb{Z}p$ as subsets of $\overline{\Pic}(E)$.
In particular, if $\Div(\wt{R})$ is countable, then so is $\Div(R)$.  
\end{proposition}
\begin{proof}
This is similar to \cite[Proposition~9.5]{Rog09}.  As noted in the last proof, 
  $\Pic( E) \to \overline{\Pic}(E)$ is a countable-to-one function and so
  the second statement of the proposition follows from the first.

Given an admissible $R$-module $M$, 
let $\wt{M} = M \otimes_R \wt{R}$ and consider the  exact sequence
\begin{equation}\label{induct-equ}
0 \to \Tor_1^R(M, \wt{R}/R) \to M \overset{\phi}{\to} \wt{M} \to M \otimes_R (\wt{R}/R) \to 0.
\end{equation}
Clearly $\wt{M} \in \rgr \wt{R}$ and it is easy to see that $\wt{M}$ is an admissible $\wt{R}$-module
(see the proof of \cite[Proposition~9.5]{Rog09}(3)).  Thus $\wt{M}$ is also admissible over $R$ and, by Lemma~\ref{C-lem},
 $\Div_R\wt{M}=\Div_{\wt{R}}\wt{M}$.   
The outer terms of \eqref{induct-equ} are isomorphic as right $R$-modules to subfactors of a finite direct sum  $(\wt{R}/R)^{(r)}$, and 
thus the kernel and cokernel of $\phi$ are also isomorphic to admissible $R$-module subfactors of  some $(\wt{R}/R)^{(s)}$.    
Using Lemma~\ref{lem:divprops}(4), it now suffices to prove that  such a module $N$ satisfies $\overline{\Div}(N) \in \mb{Z}p$.  
 By Proposition~\ref{prop:subtle}, $\wt{R}/R \cong \bigoplus_{i \geq 0} L[-i-1]$ for a $g$-torsionfree
line module $L$.   Thus $N$ has a finite filtration by subfactors of shifts of $L$, and it  suffices to prove 
that $\overline{\Div}(N) \in \mb{Z}p$ for a subfactor $N$ of a shift of $L$.  By  Lemma~\ref{lem:divprops}(3) 
 $\overline{\Div}(N[i]) = \overline{\Div}(N)$, so  we may even assume that $N$ is a subfactor of $L$.   By Lemma~\ref{lem:crit}(2) and Proposition~\ref{prop:subtle},  $L$ is  $2$-critical with $\Div L  =  \tau(p)$.  Thus, by Lemma~\ref{lem:divprops}(4,5),
   either $\overline{\Div} N = 0$ or $\overline{\Div} N = \overline{\Div} L=[p]$ in $\overline{\Pic}(E)$, and we are done.  
\end{proof}

  We now apply the countability of $\Div T(\divd)$ to study line modules and \spe\ ideals.
  Since we may need to  extend the base field, we will need that $\Div (R\otimes_{\kk}K)$ is countable for some 
  uncountable, algebraically closed  field $K\supseteq \kk$; in which case we call $\Div(R)$ \emph{strongly countable}. 
  \label{countable-defn}  Of course, this is automatic if 
  either $\kk$ is uncountable or, by the last two results,   if $R$ is constructed from a Sklyanin-like algebra.

\begin{lemma}\label{lem:finitediv}
Let $R = T(\divd)$ for some effective $\divd$ with $\deg \divd  < \mu$.  Assume that $\Div(R)$ is strongly countable.
Then the set $\big\{ \Div L \st L_R \mbox{ is a $g$-torsionfree cyclic line module} \big\}$ is finite.
\end{lemma}

\begin{proof}
This is very similar to the proof of \cite[Theorem~9.7]{Rog09}, the main difference being that   line modules are automatically $g$-torsionfree 
and cyclic in the context of that paper.  In contrast, 
Example~\ref{ex5}  shows that line modules can even be $g$-torsion in our case.

The ring $R$ is strongly noetherian by Theorem~\ref{thm:Rnoeth} and so, by  \cite{AZ2}, isomorphism classes of cyclic line modules over $R$ are parametrised by a projective scheme, the \emph{line scheme}.  More specifically, for each $m$ the set of graded subspaces 
$\bigoplus_{j=0}^m J_j\subseteq \bigoplus_{j=0}^m R_j$  
 such that $\dim J_i = \dim R_i - i$ for each $i$, and such that $J_i R_j \subseteq J_{i+j}$ for all $i,j$ with $i + j \leq m$, 
is parametrised by a projective scheme $X_m$.  There is a truncation map $\phi_m: X_{m+1} \to X_m$ for each $m$ 
induced by forgetting the $(m+ 1)^{\text{st}}$ coordinate, and \cite[Corollary~E4.5]{AZ2} implies that there is $m_0 \in \NN$ so that $\phi_m$ is an isomorphism for $m \geq m_0$.  Then  $X_{m_0}$ is the line scheme.

If $L = R/J$ is a line module, then $L/Lg$ has Hilbert series at least as large as the Hilbert series $(1-t)^{-1}$
 of a point module; and it has the Hilbert series of a point module if and only if $L$ is $g$-torsionfree.
As we run over the set of $J$'s such that $R/J$ is a line module, $\dim_\kk J_m + gR_{m-1}$ is a lower semi-continuous function 
with maximum value $\dim_\kk R_m - 1$.  Thus requiring   $\dim_\kk J_m + gR_{m-1} = \dim_\kk R_m - 1$ for any fixed $m$ is a further open  
condition on the set of such $J$.   Now apply \cite[Corollary~E4.5(2)]{AZ2} to the Hilbert series of a point module.   It shows that there is $n$ such 
that for any right ideal $I$ of $R$ such that $\dim_{\kk} (R/I)_m = 1$ for all $0 \leq m \leq n$, then $\dim_{\kk} (R/I)_m = 1$ for all $m \geq 0$ (and 
so $R/I$ is a point module). This implies that the condition that a line module $R/J$ be $g$-torsionfree can be tested in a fixed finite set of degrees.  
Thus it is an intersection of finitely many open conditions, and so the set of $g$-torsionfree line modules is an open subset $W$ of the line scheme; 
in particular 
it is a $\kk$-scheme of finite type.

The rest of the proof is basically the same as  \cite[Theorem~9.7]{Rog09}, simply using $W$ 
in place of the entire line scheme, but we sketch some of the details for the reader's convenience.  As in 
 that proof,  we can extend the base field if necessary and 
  assume that $\kk$ is uncountable.    There is a map $\theta: W \to E$  sending a $g$-torsionfree cyclic line module $L = R/J$ to its divisor 
$\Div L$, which is the single point $p\in E$ corresponding to the point module $P(p)=L/gL$.  For any fixed   $q \in E$, $\Div L = q$ 
if and only if $J + gR = I$, where $P(q)=R/I$.  Thus 
the fibre $\theta^{-1}(q)$  is a closed subset in $W$.
Since $\Div T$ is countable, $W$ is therefore a union of countably many proper closed subsets. By, for example,  \cite[Lemma~6.5]{NS}
this forces the  union to be finite.  Equivalently, 
 the set of divisors associated to $g$-torsionfree cyclic line modules  must be finite.
\end{proof}

 \begin{remark}\label{uncountable}  Note that $\Div T$  can be uncountable for  algebras satisfying Hypothesis~\ref{hyp:main}.
For example, consider $T=B(E,\mc{M},\tau)[g]$, where $g$ is an indeterminate in degree one and, as usual, $\deg\mc{M}\geq 2$.  For each point $p \in E$ and corresponding right point module $P(p) = B/I$, the $T$-module $L = T/I[g]$ is a $g$-torsionfree line module with $\Div L = p$.
Hence $\{\Div L \,|\, L$ is a $g$-torsionfree line module$\}$ has cardinality at least as large as that of $E$.  
Thus for this example, Lemma~\ref{lem:countable} fails if $\kk$ is uncountable. 
On the other hand, if $\kk$ is countable then it shows that Lemma~\ref{lem:finitediv} can fail if one  replaces the strongly countable hypothesis by a countable one. 
   \end{remark}

We now turn to the study of sporadic ideals. 
The punchline of the proof, Proposition~\ref{prop-endgame}, shows how the minimal \spe\ ideal for $R = T(\divd + p)$ 
is determined by the minimal \spe\ ideal for $\wt{R} = T(\divd)$, together with the structure of the 
$g$-torsionfree GK-1 factor modules of the right and left  exceptional line modules   coming from  Proposition~\ref{prop:subtle}.
  The next result shows how to  control these factors.  The first part of the next
  result is a generalisation of \cite[Lemma~10.1]{Rog09}.

\begin{lemma}\label{lem:specann}
Let $R = T(\divd)$ where $T$ satisfies Hypothesis~\ref{hyp:main} and $\deg \divd < \mu$, and assume that the set 
$\{\Div M \st M\ \text{is a $g$-torsionfree cyclic line module for}\ R \}$ is finite.      Let $L$ be any 
$g$-torsionfree (not necessarily cyclic) line module for $R$.
\begin{enumerate}
\item $L$ has a unique smallest submodule $N$ such that $L/N$ is $g$-torsionfree 
 with $\GKdim L/N=1$.
\item Consider the right annihilator $I = \operatorname{r-ann}(L/N)$.  Then $I$ is a \spe\ ideal that annihilates every 
finitely generated  $g$-torsionfree   subfactor $M$ of a  direct sum of shifts of $L$ with $\GKdim M= 1.$
\end{enumerate}
\end{lemma}
\begin{proof}

(1)  A module $N\subset L$ such that $L/N$ is $g$-torsionfree of GK-dimension 1 will be called a \emph{good submodule}.
To prove the result,    it suffices to prove that $L$ has DCC on good submodules.
  For any good submodule $N$, $\dim_{\kk}L/(N + Lg)<\infty$ and so 
  $\dim_\kk (L/N)_n$ is a constant, say $d$, 
for $n \gg 0$.  We call $d$ the \emph{multiplicity} of $L/N$.  To prove Part (1) it 
also suffices to show that the possible multiplicity of $L/N$ for good submodules $N$ is bounded above.

Set $B=R/gR$ and 
  consider $P = L/Lg\in \rgr B$.  Then $h_{P}(t)=1/(1-t)$; so $P$ is a (not necessarily cyclic) point module.    
We claim that there is $n_0 \geq 0$ such that every nonzero submodule of $P_{\geq n_0}$ is a shifted cyclic point module.
Write $R/Rg= B(E, \mc{L}, \tau)$, where $\mc{L} = \mc{M}(-\divd)$.  By 
Lemma~\ref{lem:pt-crit}(1,2), there is a point module of the form $P(p) = \bigoplus_{n \geq 0} H^0(E, \kk(p) \otimes \mc{L}_n)$ such 
that $P_{\geq m} \cong P(p)_{\geq m}$.  Thus we can replace $P$ by $P(p)$. Now if $P_m B_1 = 0$, for some $m \geq 0$, then by  the 
definition of multiplication in $B$,  the image of the multiplication map 
$f: H^0(E, \mc{L}_m) \otimes H^0(E, \mc{L}^{\tau^m}) \to H^0(E, \mc{L}_{m+1})$ is contained in 
$H^0(E, \mc{I}_p \otimes \mc{L}_{m+1})$, where $\kk(p) = \mc{O}_E/\mc{I}_p$.   If $\deg \mc{L} \geq 2$ this never happens, 
by Lemma~\ref{lem:sec-mult}.  If $\deg \mc{L} = 1$, then  by a similar argument as in Proposition~\ref{prop:subtle}, 
 this happens if and only if $\mc{L}^{\tau^m} \cong \sO_E(p)$. 
This can happen for at most one $m \geq 2$.  Thus some tail $P_{\geq n_0}$ has the property that $P_i B_1 = P_{i + 1}$ for 
all $i \geq n_0$.   In particular, every nonzero submodule of $P_{\geq n_0}$ is a shifted cyclic point module, proving the claim. 

Suppose that $N\subset L$ is a good submodule with $N_{\leq n_0} = 0$, where $n_0$ is defined by the last paragraph.
Tensoring $0 \to N \to L \to L/N \to 0$ with $R/Rg$ gives an exact sequence $0 \to N/Ng \to L/Lg \to L/(N + Lg) \to 0$, by Lemma~\ref{lem:divprops}(2).  
By definition,  $N/Ng\subseteq (L/Lg)_{\geq n_0}$, and so   is a shifted cyclic point module.  Since $N$ is $g$-torsionfree, this implies 
that $N$ is a shifted cyclic line module.  
Now $\Div L/N = 0$, since $L/N$ is $g$-torsionfree with $\GKdim(L/N)=1$.  Thus $\Div L = \Div N = p$, by Lemma~\ref{lem:divprops}(4).
   Writing $N \cong L'[-i]$ for a line module $L'$, then $\Div L' = \tau^{i}(p)$ by Lemma~\ref{lem:divprops}(3).   
We have therefore found a $g$-torsionfree cyclic line module for $R$ with divisor $\tau^{i}(p)$.  Note that $i$ is also the multiplicity of $L/N$.
Since by assumption the set of divisors associated to $g$-torsionfree cyclic line modules is finite, there is an upper bound on $i$ and thus on the multiplicity of $L/N$ for good submodules $N$ with $N_{\leq n_0} = 0$.

Now suppose that we have a descending chain $N^{(1)} \supseteq N^{(2)} \supseteq \dots$ of good submodules of 
$L$.  If 
$N^{(i)}_{\leq n_0} = 0$ for some $i$, then for all  $j>i$ the multiplicity of $L/N^{(j)}$ is bounded above by the previous paragraph, and 
thus the chain must stabilise.   Otherwise 
$N^{(i)}_{\leq n_0}\not=0$   for all $i$; but  this implies that the exists  a fixed finite-dimensional nonzero vector subspace $X$ of $L$
 such that  $N^{(i)}_{\leq n_0}  = X$ for all $i \gg 0$.   The submodule 
$XR$ of $L$ is contained in every $N^{(i)}$, and $\GKdim L/XR = 1$ since $L$ is 2-critical by Lemma~\ref{lem:crit}(2).  
Thus the multiplicity of $L/XR$  is an upper bound for the multiplicity of all   the $L/N^{(i)}$, and again the chain stabilises.

(2)  Exactly the same proof as in \cite[Lemma~10.2]{Rog09} applies.
\end{proof}

\begin{proposition}
\label{prop-endgame}
Let $T$ satisfy Hypothesis~\ref{hyp:main} and suppose in addition that $\Div(T)$ is strongly countable, 
and that $T$ has a minimal \spe\ ideal.  Let $R = T(\divc)$ where $\deg \divc < \mu$.

\begin{enumerate}
\item $\Div(R)$ and $\Div^{\ell}(R)$ are also strongly countable.

 \item $R$ has a minimal \spe\ ideal, which  can be determined inductively  as follows.  
Write $R = T(\divd + p) \subseteq \wt{R} = T(\divd)$, where $\deg \divd \leq \mu - 2$.
Suppose that $H$ is a minimal \spe\ ideal of $\wt{R}$.   Let $L$, respectively $L'$, 
 be the right, respectively left, exceptional line module for the blowup $R \subseteq \wt{R}$,  
  as in Remark~\ref{line-defn}.  Then Lemma~\ref{lem:specann}(2) holds for   $L$ on the right and $L'$ on the left, and
produces  \spe\ ideals $I$, respectively $G$.   Then $G(H \cap R)I$ is a minimal \spe\ ideal for $R$.
\end{enumerate}
\end{proposition}
\begin{proof}
(1) Take an uncountable algebraically closed field extension $K\supseteq \kk$. 
By Remark~\ref{uniqueness-remark}, $T(\divc)\otimes_{\kk} K =(T\otimes_\kk K)(\divc)$ for any divisor $\divc $ 
and so the earlier results apply to $R\otimes_{\kk}K$. 
In particular,  by Proposition~\ref{prop:iterate}, we can form $R\otimes_\kk K$ iteratively via a series of one point blowups starting with 
$T\otimes_\kk K$.   Then $\Div(R)$ is strongly countable by Proposition~\ref{prop-induct} and induction.   
By Corollary~\ref{cor:leftright} and  induction on a left-sided 
version of Proposition~\ref{prop-induct}, $\Div^{\ell}(R)$ is strongly countable.

(2) Clearly it suffices to work inductively as stated.  By  Lemma~\ref{lem:finitediv} 
  the hypotheses of Lemma~\ref{lem:specann} hold for $R$ and   $L$, and thus Lemma~\ref{lem:specann}(2) 
produces a \spe\ ideal $I$ which kills all   $g$-torsionfree subfactors $N$ of direct sums of shifts of $L$  with $\GKdim N=1$.  
The same argument on the left constructs $G$.
Now   the   proof  of \cite[Theorem~10.4]{Rog09} applies, and shows that $G(H \cap R)I$ is a minimal \spe\ ideal for $R$.  
\end{proof}

The following theorem shows that the previous result applies to the most important examples 
of interest.

\begin{theorem}\label{8-special}
Let $T = S^{(\gamma)}$, where $S$ is a generic cubic or quadratic Sklyanin algebra, with central element $g \in S_{\gamma}$.  Then
 $T$ satisfies the hypothesis of  Proposition~\ref{prop-endgame} (in fact, $T$ has no \spe\ ideals).  Thus  every blowup $R = T(\divd)$ 
 with $\deg \divd < \mu$ has a minimal \spe\ ideal.  
\end{theorem}

\begin{proof}  Let $K\supseteq\kk$ be an uncountable algebraically closed field extension. 
It is clear that $S\otimes_{\kk}K$ is still a  Sklyanin algebra, since the defining relations do not depend upon the choice of field and it is  still generic by \cite[Theorem~7.1]{ATV2}. 
Hence it  is a Sklyanin-like algebra of  global dimension 3.
 Thus $\Div(T)$ is strongly countable by Lemma~\ref{lem:countable}.
Further, \cite[Theorem~I]{ATV2} shows that $S[g^{-1}]_0 = T[g^{-1}]_0$ is simple. 
Since $T/I$ cannot be $g$-torsion for a special ideal $I$,  it follows   that $T$ cannot have a  \spe\ ideal. 
Thus   $T$ satisfies the hypotheses of Proposition~\ref{prop-endgame}.
\end{proof}

To end the section, we give an example to show that the complexities in Lemma~\ref{lem:finitediv} are in fact needed:  for the rings  
   $R=T(\divd)$ where $\divd$ has degree $\mu -1$, a line module need not be $g$-torsionfree 
   nor $2$-critical  in general.  
   If $\deg \divd < \mu-1$, then, as previously remarked, all line modules over $T(\divd)$ are $g$-torsionfree.

\begin{example}\label{ex5} 
 Let  $R=T(\divd)$ where $T$ satisfies Hypothesis~\ref{hyp:main} and $\divd$ 
has degree $\mu -1$. Then there exists a cyclic line module $M \in \rgr R$ which is $g$-torsion, and which
 has a shifted cyclic point module as a submodule. \end{example}

\begin{proof}
 By Theorem~\ref{thm:Rnoeth}(1),  $T(\divd)$ has   Hilbert series $h_{T(\divd)}(t)=  \frac{t^2-t+1}{(1-t)^3}$ and so 
      $h_B(t)=   \frac{t^2-t+1}{(1-t)^{2}}$ for $B=T/gT$.   Now consider $M=R/g^2R$.  As $g^2$ is regular of degree 2,  
  $  h_M(t)=\frac{t^2-t+1}{(1-t)^3}(1-t^2) = 
  \frac{1+t^3}{(1-t)^2}
  .$
Pick any $0\not=x\in M$ that  lies in degree 2 in the submodule $gR/g^2R  \cong B[-1].$ 
 Then $xB=xR\cong B[-2]$ has Hilbert series  $  \frac{t^2(t^2-t+1)}{(1-t)^{2}}$   and so $xB_+=xR_+$ has the  Hilbert series  
  $  \frac{t^2(t^2-t+1)}{(1-t)^{2}} - t^2 =  \frac{t^3}{(1-t)^{2}}$  of a  shifted line module.
 
 Finally, set $L=M/xR_+=R/(g^2R+xR_+)$. This  is  cyclic  with 
 $h_L(T)=   \frac{1+t^3-t^3}{(1-t)^2}
 =\frac{1}{(1-t)^2}.$   If $L'=gR/(g^2R+xR_+)\subset L$, then 
  $\overline{L}=L/L'\cong B$. Moreover  $L'  =gR/(g^2R+xR_+)$  is also   cyclic  with 
$h_{L'}(t)=  \frac{t}{(1-t)}$, so $L'$  is a shifted point module. \end{proof}

%%%%%%%%%%%%%%%%%%%%%%%%%%%%%
%%%%%%%%%%%%%%%%%%%%%%%%%%%%%

\section{On the nonexistence of \spe\ ideals}\label{SPECIAL}

In Theorem~\ref{8-special} we exhibited  a large number of algebras with a minimal special ideal, and this leads to a natural question:
when do these algebras $T$ have no \spe\ ideal?    The question is delicate in general, and  a partial answer is given in \cite[Section~11]{VdB}
  when $T = S^{(3)}$ for the quadratic Sklyanin algebra $S$, but the answer  uses much 
 of the machinery of that paper.  In this section, 
we give an elementary argument which can be used to answer the question 
in a few easy cases.  We   restrict our attention to the following situation.

\begin{notation}\label{sec9-notation} Throughout the section, 
we fix a ring $T$ satisfying Hypothesis~\ref{hyp:main}.  
Let $\divd$ be an effective divisor on $E$ with $\deg \divd < \mu -2$, and let $p \in E$.
Set $\wt{R} = T(\divd)$ and $R = T(\divd + p)$; thus $R \cong \wt{R}(p)$ by Proposition~\ref{prop:iterate}. 
Also, by Theorem~\ref{thm:Rnoeth}, $\wt{R}$ and $R$ are generated in degree 1.
Throughout   $M(k, p)= M_{\wt{R}}(k,p)$ in the notation of  Definition~\ref{def:M(k,D)} and 
 $M'(k, p)= M'_{\wt{R}}(k,p)$ in the notation of  Definition~\ref{def:M'(k,D)}.
\end{notation}

\begin{proposition}
\label{prop:nospecial} Keep  $R = \wt{R}(p) \subseteq \wt{R}$ as in Notation~\ref{sec9-notation} 
and assume that $\wt{R}$ has no \spe\ ideals and that $\Div(\wt{R})$ is strongly countable.
Suppose in addition that: 
\begin{enumerate}
\item[(C1)] for   $q \in \{ \tau^i(p) \st i \leq 1 \}$, there exist $k \geq 1$ and $K \subseteq \wt{R}_k$ with 
$K \wt{R}_1 = \wt{R}(q)_{k+1} = M(k+1, q)_{k+1}$;
\item[(C2)] for  $q \in \{ \tau^{-i}(p) \st i \leq 1 \}$, there exist $k \geq 1$ and $K \subseteq \wt{R}_k$ with $\wt{R}_1 K = \wt{R}(q)_{k+1} = M'(k+1, q)_{k+1}$.
\end{enumerate}
Then $R$   has no \spe\ ideals.

\end{proposition}
Before proving the proposition,  we  make some comments about its statement  and give a lemma.

\begin{remark}
\label{rem:nospecial}
\begin{enumerate}
\item
Since $M(k+1, q)$ is a saturated right ideal of $\wt{R}$, if  condition (C1) holds for $K$ then necessarily 
$K \subseteq M(k + 1, q)_k$ and so the condition is also satisfied by $K'=M(k+1,q)_k$. 
Thus (C1) is equivalent to saying that   $M(k + 1, q)$ can be generated in  one degree lower than is necessary (it is generated in degree $k + 1$ by  Lemma~\ref{lem:Mmult}).  
\item If (C1) holds for some $k$ and $K$, then we can also satisfy (C1) for any larger $k$; indeed, 
if $U = \wt{R}(q)$ then   $[U_m K ]\wt{R}_1 = U_m U_{k+1} = U_{k + 1 + m}$ for any $m \geq 0$. 
\end{enumerate}
\end{remark}

\begin{lemma}
\label{lem:compress}
Let $\wt{R}=T(\divd)$ as above  
  and  assume that (C1) and (C2) hold. 
\begin{enumerate}
\item For any $\ell \geq 1$, there is $m \geq 1$ 
and $K \subseteq \wt{R}_m$ such that $K \wt{R}_{\ell} = R_{\ell + m}$.  
\item For any $\ell \geq 1$ there is $m' \geq 1$ and $K' \subseteq \wt{R}_{m'}$ 
such that $\wt{R}_{\ell} K' = R_{\ell + m'}$.
\end{enumerate}
\end{lemma}

\begin{proof}
By symmetry, we need only prove   Part~(1).  When $\ell = 1$, this is just (C1) with $q = p$.  So, by induction,
suppose the result holds 
for some $\ell\geq 1$, say for  $K \subseteq \wt{R}_m$.  Choose $k \geq 1$ such that  (C1) holds for 
 $q = \tau^{-\ell}(p)$.  By Remark~\ref{rem:nospecial}(1), this is equivalent to   
$M(k +1, \tau^{-\ell}(p))_k \wt{R}_{1} = M(k+1, \tau^{-\ell}(p))_{k+1}$.
Now  Corollary~\ref{cor:multcomp} implies that
\[
M(k+1, \tau^{-\ell}(p))_{k+1} \wt{R}_{\ell} = M(k+1, \tau^{-\ell}(p))_{k+\ell + 1} = \wt{R}_{\ell} M(k+1, p)_{k+1}.
\] 
Finally,   we calculate
\begin{align*}
K M(k + 1, \tau^{-\ell}(p))_k \wt{R}_{\ell +1} \ =&\  K M(k + 1, \tau^{-\ell}(p))_k \wt{R}_1 \wt{R}_{\ell}\ = \ 
K M(k+1, \tau^{-\ell}(p))_{k+1} \wt{R}_{\ell} \\
& =\ K \wt{R}_{\ell} M(k+1, p)_{k+1} \ = \  
R_{\ell+m} R_{k+1} \ = \  R_{\ell + m + k + 1},  
\end{align*}
proving the induction step.
\end{proof}

\begin{proof}[Proof of Proposition~\ref{prop:nospecial}]
By Proposition~\ref{prop-endgame} applied  to the ring $\wt{R}$, 
 the minimal \spe\ ideal of $R$ has the form $G(H \cap R)I$ where $H$ is the minimal \spe\ ideal of $\wt{R}$,  
 and $G$ (respectively $I$) is the \spe\ ideal of $R$ which kills all $g$-torsionfree proper factors of the left (respectively right) 
  exceptional line module for the blowup $R \subseteq \wt{R}$.
By hypothesis,   $H=\wt{R}$ and so, by symmetry, it suffices 
to  prove that 
$I = R$.  

Let $L = R/J$ be the  exceptional right  line module for the blowup $R \subseteq \wt{R}$. By   Proposition~\ref{prop:subtle}(2,3), 
 $L$ is cyclic and $J = \bigoplus_{n \geq 0} M(n+1, \tau(p))_n$.   
Let $x \in R_n \setminus J_n$ for some $n \geq 0$.  If we show that $R/(J + xR)$ is $g$-torsion for all such $x$, then $L$ will have no 
$g$-torsionfree proper factors and we will be done.   It is enough to   prove this for all 
$n \gg 0$, since any nonzero submodule of $R/J$ will contain elements of all large degrees.

Now we use condition (C1) for  $q = \tau(p)$ and  some fixed $k$.    
By Remark~\ref{rem:nospecial}(1,2),    for  $n \geq k$  
the right $\wt{R}$-module $M(n+1, \tau(p))$ is generated in degrees $\leq n$.  Thus 
$J_n \wt{R} = M(n+1, \tau(p))_{\geq n}$ if $n \geq k$.

Let   $V = R_n\wt{R}/J_n\wt{R}$ for some $n \geq k$.  Then
$R_n\wt{R} = M(n,p)_n \wt{R} = M(n,p)_{\geq n}$, using Lemma~\ref{lem:Mmult}.  Hence 
 $V \cong [M(n, p)/M(n+1, \tau(p))]_{\geq n}$ and so, 
by Theorem~\ref{thm:vdb},
$\wh{\pi(V)}_p$ is isomorphic 
to a left-infinite row vector 
\[
( \dots, 0, 0, \dots, \overbrace{k[[x]]/(x), \dots, k[[x]]/(x)}^{n+1})
\]
with the obvious right $C_p$-action.   This is a non-zero uniserial $C_p$-module, and so 
 $\pi(V)$   is  a uniserial object  in $\rqgr \wt{R}$.  Since  $x \in R_n \setminus J_n$, the $\wt{R}$-module
  $W= (x\wt{R} + J_n)/J_N\wt{R} \subseteq V$ is non-zero and   hence   is infinite dimensional   since $M(n+1, \tau(p))$ is saturated.  
 Thus $ \pi(W) \not=0$  
   and so  must  contain the simple socle of $\pi(V)$.  This 
socle is easily seen to be $\pi[(g^n\wt{R} + J_n\wt{R})/(J_n \wt{R})]$.    Hence
$  g^n \wt{R}_{\ell} \subseteq x\wt{R}_{\ell} + J_n \wt{R}_{\ell}$ for some $\ell$.

Now use Lemma~\ref{lem:compress}(2) to write $\wt{R}_{\ell} K = R_{\ell + m}$, for some 
$K \subseteq \wt{R}_m$.  Then  $J_n \wt{R}_{\ell} K + x\wt{R}_{\ell} K \supseteq g^n \wt{R}_{\ell} K$, 
or in other words $J_n R_{\ell + m} + xR_{\ell +m} \supseteq g^n R_{\ell + m}$.  Thus
 $R/(J + xR)$   is $g$-torsion, and we are done.
\end{proof}

Now we give a few examples where the criteria of the proposition can be checked. Many of these 
results can also be obtained from \cite[Section~11]{VdB}, but we prefer to give elementary proofs.

\begin{example} \label{example1}
 Let $T = S^{(3)}$ for  a generic quadratic Sklyanin algebra $S$.  Then  $T(p)$ has 
no \spe\ ideals.  
\end{example}

\begin{proof}  By Theorem~\ref{8-special} and its proof, $T$ has no \spe\ ideals and  $\Div(T)$ is strongly countable.
Recall that   $S/gS \cong B(E, \mc{L}, \sigma)$, and so $T/gT \cong B(E, \mc{M}, \tau)$ where $\tau = \sigma^3$ and
 $\mc{M} = \mc{L}\otimes\mc{L}^\sigma \otimes \mc{L}^{\sigma^2}$.
For $q \in E$, the subspace $W(q) = H^0(E, \mc{L}(-q)) \subseteq H^0(E, \mc{L}) = S_1$ is called \emph{a point space}.
  By \cite[Lemma~4.2]{Rog09},  $W(q) S = I$, where $S/I$ is the point module corresponding to
 $q$.  In particular, 
$W(q) S_2 = M(2, q)_1 = T(q)_1$.  (Note that a subscript on a graded piece of $S$ indicates the $S$-degree, 
but a subscript on a   right 
ideal of $T$ indicates the $T$-degree.)  
We also have $W(q) S_1 = S_1 W(\sigma(q))$ by   \cite[Lemma~4.1(1)]{Rog09}, or by
Lemma~\ref{lem:sec-mult}.
Then for any $q$,  
\[
W(q) W(\sigma^{-2}(q)) S_1 S_3 = [W(q) S_2] [W(q) S_2] = T(q)_2 = M(2, q)_2.
\]
This shows that  (C1), and by symmetry (C2),  holds   for all $q \in E$.   
Now apply Proposition~\ref{prop:nospecial}.
\end{proof}

\begin{example}\label{example2}
 Let $T = S^{(3)}$ as in Example~\ref{example1}.  Suppose that $p$ and $r$ are points on distinct $\tau$-orbits.  
Then   $T(p + r)$ has no \spe\ ideals.
\end{example}

\begin{proof}  By \cite[Lemma~4.1(2)]{Rog09} or   Lemma~\ref{lem:sec-mult}, 
the  point spaces satisfy  
$W(p) W(\sigma(q)) = W(q) W(\sigma(p))$ whenever $p \neq \tau(q)$ and $q \neq \tau(p)$.

Write $\wt{R} = T(r)$ and $R = \wt{R}(p) = T(p + r)$.    Note that $\wt{R}_1 = W(r)S_2$ as in the previous example and 
that $\wt{R}(q)_1 =  W(q)W(\sigma(r))S_1$ by \cite[Lemma~4.6(1)]{Rog09}.  
Then for any $q$ on the $\tau$-orbit of $p$,   
\begin{gather*}
[W(q)W(\sigma^{-2}(q))W(\sigma^2(r))]\wt{R}_1 
= [W(q)W(\sigma(r))W(\sigma^{-1}(q))][W(r)S_2] \\
= [W(q)W(\sigma(r))S_1][W(q) W(\sigma(r))S_1] 
=  \wt{R}(q)_1 \wt{R}(q)_1 = \wt{R}(q)_2.
\end{gather*}
  Thus (C1) holds, 
and the proof of (C2) is analogous.
By Example~\ref{example1} $\wt{R}$ has no \spe\ ideals, and   $\Div(\wt{R})$ is strongly countable by Proposition~\ref{prop-induct}, so
the result follows from Proposition~\ref{prop:nospecial}.\end{proof}

It is not hard to show on the other hand that $T(p + \sigma^i(p))$ always has a \spe\ ideal if $i \neq 0$ (see \cite[Proposition~11.2(1)]{Rog09}).   We expect that $T(2p)$ has no \spe\ ideals, though
 Proposition~\ref{prop:nospecial} does not apply.

\begin{example}
 Suppose  instead that $T = S^{(4)}$ for  the generic cubic Sklyanin algebra $S$, with central element $g \in S_4$.    
Then $T(p)$ has no \spe\ ideals for any point $p\in E$.
\end{example}

\begin{proof} The proof is similar to the previous examples.  Again, $T$ has no \spe\ ideals, and $\Div(T)$ is strongly countable, by Theorem~\ref{8-special}. 
For $q\in E$, set
$V(q)  = H^0(E, \mc{L}_2(-q)) \subseteq H^0(E, \mc{L}_2) = S_2$.  The rule 
$V(q)S_1 = S_1V(\sigma(q))$ follows from Lemma~\ref{lem:sec-mult} while, $T(q)_1 = V(q) S_2$ follows from  \cite[Proposition~6.7]{ATV2}.
Then  
\[
[V(q)V(\sigma^{-2}(q))]T_1 = V(q)S_2V(q) S_2 = T(q)_2.
\]
   Since this holds for all $q$, condition (C1) is verified and (C2) is similar.  
Now apply Proposition~\ref{prop:nospecial}.  
\end{proof}

 \section*{Index of Notation}\label{index}
\begin{multicols}{2}
{\small  \baselineskip 14pt

$\ppe $, equal in high degree \hfill\pageref{ppe-defn}

$\alpha$-pure  \hfill\pageref{pure-defn}

Admissible module  \hfill\pageref{def-admiss}

Allowable divisor layering  $\divd^{\bullet}$   \hfill\pageref{allowable}

Auslander Gorenstein and CM  \hfill\pageref{AG-defn}

 $\divc(M) $ divisor  \hfill\pageref{def-C}
 
 $\Div$, $\overline{\Div}$ , $\overline{\Pic}$  \hfill\pageref{def-admiss}
 
 $\Div^\ell$, $\overline{\Div}^\ell$    \hfill\pageref{def-admiss2}
 
 $\divd_n=\divd+\divd^\tau+\cdots+\divd^{\tau^{n-1}}$
  for a divisor $\divd$  \hfill\pageref{def:M(k,D)}

$E$ elliptic curve  \hfill\pageref{hyp:main}

Exceptional line modules \hfill\pageref{line-defn}

Functors $F_q$, $ G_{\divd}  $ \hfill\pageref{F-defn}

$g$-divisible \hfill\pageref{g-div-defn}

Line module \hfill\pageref{line-defn1}

$\mathfrak{J}(\divd^\bullet)$, $\mathcal{J}(\divd^\bullet)$, $J(\divd^\bullet)$
 \hfill\pageref{def:J}}

$M(k,\divd)$ \hfill\pageref{def:M(k,D)}

$\mu=\deg\mathcal{M}$    \hfill\pageref{hyp:main}

Point module, shifted point module  \hfill\pageref{point-defn}

$\pi: \rgr R\to\rqgr R $  \hfill\pageref{section-defn}

$Q(i,r,d,p)$   \hfill\pageref{def-Q}

Section functor $\omega$,   \hfill\pageref{section-defn}

Saturation  $I^{sat}$,  saturated right ideal  \hfill\pageref{lem:T-props}

 Sklyanin algebras    \hfill\pageref{hyp:main}

Sklyanin-like algebra   \hfill\pageref{def-C}

Sporadic ideal, minimal sporadic ideal  \hfill\pageref{sporadic-defn}

Strongly countable   \hfill\pageref{countable-defn}

$T$, $T/gT\cong B(E,\mc{M},\tau)$ \hfill\pageref{hyp:main}

$T(\divd)$,  $T_{\ell }  \ast T(\divd)$ \hfill\pageref{def-blowup}
  
\end{multicols}

  %%%%%%%%%%%%%%%%%%%%%%%

\bibliographystyle{amsalpha}

%\bibliography{../../../Dropbox/biblio}

\providecommand{\bysame}{\leavevmode\hbox to3em{\hrulefill}\thinspace}

\providecommand{\MR}{\relax\ifhmode\unskip\space\fi MR }


\providecommand{\MRhref}[2]{%

  \href{http://www.ams.org/mathscinet-getitem?mr=#1}{#2}

}

\providecommand{\href}[2]{#2}

\end{document}